\documentclass[11pt]{article}
\usepackage{t1enc}
\usepackage[utf8]{inputenc}
\usepackage[english]{babel}
\usepackage{amsmath,amssymb,amsfonts,amsthm,}
\usepackage{array} 
\usepackage[nottoc]{tocbibind}
\usepackage{url}
\usepackage{times,mathptmx} 
\usepackage{enumitem} 
\usepackage{algorithm2e}

\setlist[enumerate,1]{label=\textbf{\arabic*.},itemsep=1pt,topsep=0pt,align=right}
\setlist[enumerate,2]{label=\textit{\alph*)},itemsep=1pt,topsep=0pt,align=right}
\setlist[itemize]{itemsep=1pt, topsep=0pt} 

\usepackage[unicode,breaklinks,colorlinks=true,allcolors=black]{hyperref} 
\usepackage{xcolor} 
\usepackage{geometry} 
\usepackage{graphicx} 
\graphicspath{ {./images/} } 

\usepackage{listingsutf8}
\usepackage{color}
\definecolor{pink}{rgb}{1,0,1}

\definecolor{cwgreen}{rgb}{0.1, 0.55, 0.1}

\newcommand{\R}{\mathbb{R}}
\newcommand{\N}{\mathbb{N}}
\newcommand{\gconv}{\xrightarrow{\ \gamma\ }}
\newcommand{\gGconv}{\xrightarrow{\ \gamma_G\ }}
\newcommand{\Mconv}{\xrightarrow{\ M\ }}
\newcommand{\ddiv}{\operatorname{div}}

\newcommand{\ltfrac}[2]{\mbox{\large$\frac{#1}{#2}$}}

\usepackage{multirow}

\usepackage{mathtools}

\newcommand{\divergence}{\operatorname{div}}
\newcommand{\innerproduct}[2]{\left\langle #1, #2 \right\rangle}

\newcommand{\vertiii}[1]{{\left\vert\kern-0.25ex\left\vert\kern-0.25ex\left\vert #1
		\right\vert\kern-0.25ex\right\vert\kern-0.25ex\right\vert}}

\makeatletter
\def\th@plain{%
	\thm@notefont{}
	\itshape 
}
\makeatother

\def\beq{\begin{eqnarray}}
\def\eeq{\end{eqnarray}}

\def\bbs{{S}_{h, {\Omega_t} }  }
\def\bbq{{M}_{h, {\Omega_t} }  }

\def\cb{{\bf c}}
\def\corr{corresponding}

\def\er{\mathbb{R}}
\def\iv{^{-1}}
\def\nb{\nabla}
\def\resp{respectively}

\def\ve{\varepsilon}
\def\vp{\varphi}
\def\vr{\varrho}

\newcommand{\he}{H^1_0(\Omega)}
\newcommand{\het}{H^1_0(\Omega_t)}

\newcommand{\B}{\overline{B}}

\def\la{\langle}
\def\ra{\rangle}

\def\sm{\smallskip}
\def\bi{\bigskip}
\def\me{\medskip}

\theoremstyle{plain}

\newtheorem{theorem}{Theorem}[section]
\newtheorem{proposition}[theorem]{Proposition}
\newtheorem{corollary}[theorem]{Corollary}

\theoremstyle{definition}

\newtheorem{definition}[theorem]{Definition}
\newtheorem{remark}[theorem]{Remark}
\newtheorem{example}[theorem]{Example}
\newtheorem{lemma}[theorem]{Lemma}
\newtheorem*{proof*}{Proof}
\newtheorem*{proposition*}{Proposition}
\newtheorem*{corollary*}{Corollary}
\newtheorem*{explanation*}{Explanation}
\newcommand{\keywords}[1]{{\renewcommand{\thefootnote}{}%
  \footnotetext{\emph{Keywords:} #1}}}
\newcommand{\subjclass}[1]{{\renewcommand{\thefootnote}{}%
  \footnotetext{\emph{2020 Mathematics Subject Classification:} #1}}}

\makeatletter
\def\thm@space@setup{%
	\thm@preskip=4pt
	\thm@postskip=4pt
}
\makeatother

\makeatletter
\let\origmaketitle\maketitle
\renewcommand{\maketitle}{\begingroup\def\@fnsymbol##1{}\origmaketitle\endgroup}
\makeatother

\title{\vspace{-1.8cm}\bfseries Continuous meets discrete: geometry dependence of the
performance of iterative solutions for nonsymmetric PDEs
in terms of Laplace eigenvalues}

\author{F. Francesconi$^{1}$
\and J. Kar\'atson$^{2,3,*}$
\and B. Lados$^{2}$
\and J. Rowlett$^{1,\dagger}$%
\thanks{$^1$Mathematical Sciences, Chalmers University and the University of Gothenburg, 41296 Gothenburg, Sweden.}%
\thanks{$^2$Department of Applied Analysis and Computational Mathematics \& HUN-REN--ELTE Numerical Analysis and Large Networks Research Group, ELTE E\"otv\"os Lor\'and University, Budapest, Hungary.}%
\thanks{$^3$Department of Analysis and Operations Research, Budapest University of Technology and Economics, Budapest, Hungary.}%
\thanks{$^*$\texttt{kajkaat@caesar.elte.hu}, \quad $^\dagger$\texttt{julie.rowlett@chalmers.se}}%
}

\date{}

\begin{document}

		
	

\maketitle


\begin{abstract}
It is a classical fact in spectral geometry that the eigenvalues of the Dirichlet
Laplacian depend continuously on the geometry of the underlying domain, under suitable geometric
restrictions. We show that this theoretical property is reflected in a concrete
numerical situation: the convergence rate of iterative solvers for
convection-diffusion problems remains stable under deformations of the domain. The connection is a superlinear estimate for preconditioned GMRES/GCR in which the
relative residual errors are bounded by an explicit expression of the eigenvalues of the
Laplacian, or of a symmetric positive elliptic operator in divergence form. We prove
that this bound depends continuously on the domain, uniformly in the iteration index,
under either regularity hypotheses on the boundary or a topological constraint in dimension two. We also establish   corresponding statements on the discrete level.
Two examples delimit the theory: domains bounded by a fractal curve, where continuity
still holds, and perforated domains, where it fails. The theoretical results, including both  examples with rough domains, are confirmed by numerical experiments.
\end{abstract}

\section{Introduction} 
The eigenvalues of the Laplacian are ubiquitous quantities in the theory and applications of partial differential equations (PDEs). 
Depending on the given models, as is well-known, they are related to
the frequencies of a vibrating membrane, the energy levels of a quantum particle, or the decay rates in heat conduction, and many other applications. Within the theory of 
eigenvalues, it is often a relevant question how sensitive the eigenvalues are to the spatial domain where the PDE is considered, and more specifically, whether they vary continuously with respect to the geometric deformations of the spatial domains.
This question has a long history, and the answer depends on the notion of convergence of domains one adopts; we refer to \cite{henrot,buttazzo} for a general theory, and to \cite{mosco,chenais,sverak1} for the results we shall use. Such results live on the \emph{continuous level}, that is, on the level of the exact PDEs.
 
In practice, solving PDEs is executed using numerical approximation methods, since the exact solution is not attainable with the exception of a few classical special cases.  These numerical methods generally consist of 
a proper discretization of the PDE and then the iterative solution of the arising algebraic problem. In the area of elliptic PDEs, i.e. of 
 Laplacian-related stationary problems, an important class of models 
 is formed by convection-diffusion type equations
 \cite{ESW, Morton, Stynes}, which collect different physical phenomena and lead to nonsymmetric problems in the sense that no energy inner product is generated. In particular, solving more compound fluid dynamical models  of Navier-Stokes type also often involve convection-diffusion  equations as building blocks.
 
 We consider such models   in this paper, where the central point is that the convergence analysis of iterative solution methods leads to characteristic quantities involving the eigenvalues of the Laplacian (or, in some more general situations, of a symmetric positive elliptic operator).
 
More precisely, the relative residuals of the preconditioned GMRES/GCR iteration obey  superlinear bounds involving   the eigenvalues $\lambda_j(\Omega)$  of the Dirichlet Laplacian, or of a symmetric positive elliptic operator in divergence form,  on $\Omega$, which thus carry geometric information of these bounds. 
The spatial domain in these problems is often subject to some perturbation or approximation.
 This may happen in various situations: whenever a curved or rough boundary is replaced by a polygon in order to be meshed, or when  the geometry itself is the unknown (as in shape optimisation, where the same solver is called on several slightly different domains), or when the boundary is only known up to manufacturing or measurement tolerances.
  Consequently, if we wish to ensure that the convergence behaviour  of these iterative methods is stable in this respect, that is, they perform in a robust way under the 
 deformations of the spatial domains, then we are led to the study of the geometric domain dependence of the Laplace eigenvalues. 
 
Altogether, we find that the analysis of a \emph{discrete} question (i.e. the performance of  iterative methods for discretized convection-diffusion type equations) leads to a study on the \emph{continuous level} (i.e. related to the operators in the PDEs). The goal of this paper is to show that this  meeting of the continuous and discrete levels is truly fruitful.
Results on the continuous level turn out to have faithful discrete analogues, which we prove for the eigenvalues and observe numerically for the convergence curves themselves.

	The paper is structured as follows. Section \ref{seccvd} presents  the
	considered class of  convection-diffusion type problems, and describes
	how the analysis of the iterative solution leads to  bounds in terms of  the eigenvalues of the Laplacian.
	Section \ref{s:continuity_geometry} first discusses the relevant notions of geometric convergence of domains, and then we give geometric conditions under which the eigenvalues depend continuously on the domain: convexity, uniform Lipschitz regularity, or, in the plane, a bound on the number of connected components of the complement. We also show that small $C^1$ deformations of a fixed domain fall into this framework, further, that all our results extend to uniformly elliptic operators in divergence form. 
	In particular, we include two extreme examples:  continuous dependence still holds
	for a limit  domain whose boundary is a fractal, whereas it already fails 
	on a proper sequence of domains perforated by shrinking disks. 	
	Section \ref{seccdddiscr} gives some discrete counterparts of the continuity results, both for a smoothly deformed finite element space   and for a sequence of refined discretizations on approximating polygons.
	Section \ref{secnumexp} presents  numerical experiments on the performance of the considered iterative methods. They illustrate 	that the
	continuous dependence on the geometric deformations is preserved from the PDE level in all the discussed situations, and finally, that the violation of continuous dependence for  the perforated domain counterexample  is also reflected in the discrete case.
	 We conclude by indicating how our results might apply beyond convection-diffusion problems to further ones with a similar operator structure.

\section{Iterative solutions of convection-diffusion type problems:  characterizing convergence with eigenvalues}
\label{seccvd}

In this section, we present how the analysis of the numerical solution of convection-diffusion models leads to the eigenvalues of the Laplacian. We first give a brief theoretical background of the studied convection-diffusion problems and their finite element solution, then we discuss the GMRES/GCR iterative methods for solving the discretized problems, together with an accurate (superlinear) estimation of their convergence. Further estimates then establish a connection between these  convergence bounds and the eigenvalues of the Laplacian. 
Finally the appearing inequalities are analyzed with theoretical calculations and numerical tests, showing that the mentioned connection
 to the eigenvalues of the Laplacian is descriptive.

\subsection{The convection-diffusion problem and its discretization}

Let $\Omega \subset \mathbb{R}^2$ be a bounded domain with piecewise smooth boundary,  $f: \Omega\rightarrow\mathbb{R}$ and $\textbf{w}: \overline{\Omega}\rightarrow\mathbb{R}^2$  be given functions.
The following boundary value problem arises in modeling convection-diffusion phenomena, often as a building block of more compound stationary flow models:
\begin{equation}
	\label{nonsym}
	\begin{cases}
		-\Delta u + \textbf{w}\cdot\nabla u = f, \\
		u|_{\partial\Omega} = 0.
	\end{cases}
\end{equation}
Suppose that the following assumptions are satisfied:
\begin{enumerate}[label=(\roman*)]
	\item $\textbf{w}\in C^1(\overline{\Omega},\mathbb{R}^2)$ and $\text{div } \textbf{w} = 0$;
	\item $f \in L^2(\Omega)$.
\end{enumerate}
Then \cite{ESW}  problem  (\ref{nonsym}) has a unique weak solution in the Sobolev space $H^1_0(\Omega)$. 

Although the above problem is our focus, for more generality 
we will consider a wider case where the principal part may include a variable matrix coefficient $G\in L^\infty(\Omega; \, \mathbb{R}^{2\times 2})$
which is uniformly s.p.d. That is, there exists $m > 0$ such that $\lambda_{min}(G(x))\geq m$ (a.e. $x\in\Omega$),
or in other words,
\begin{equation}
\label{eq:nonsymG_below}
	  G(x)\xi\cdot  \xi \ge m|\xi|^2 \qquad a.e.\  x\in \Omega, \ \xi\in \mathbb{R}^2. 
\end{equation}
Thus the PDE in  
\eqref{nonsym} becomes
\begin{equation}
\label{eq:nonsymG_op}
 	-\divergence(G\, \nabla  u) + \textbf{w}\cdot\nabla u = f \, .
\end{equation}

These conditions also allow a unique weak solution $u \in H^1_0(\Omega)$. 
We can obtain the numerical solution of the problem with the finite element method (FEM).  Consider a suitable finite dimensional subspace $V_h \subset H^1_0(\Omega)$. Then the approximate solution of the BVP is the unique solution $u_h \in V_h$ of the following discrete problem:
\begin{equation}
	\label{fem-form}
	\int_\Omega \left(G\, \nabla  u_h \cdot \nabla v_h + \left(\textbf{w}\cdot\nabla u_h\right)v_h\right) =
	\int_\Omega fv_h \qquad (\forall v_h \in V_h)\text{.}
\end{equation}
Here, to enable the inclusion $V_h \subset H^1_0(\Omega)$, it is typically assumed that   $\Omega$ itself is already a polygonal domain, that is, if the original domain had a curved boundary then its polygonal approximation has been done before applying the FEM.
Let us fix a basis in $V_h$ and express $u_h$ as the linear combination of the basis functions $\{\phi_j\}_{j=1}^N$ such that
\begin{equation}
    \label{fem-coeff}
    u_h = \sum\limits_{j=1}^{N} c_j \phi_j
    \text{,}
\end{equation}
where the coefficients $\{c_j\}_{j=1}^N$ are the unknowns to be determined. If we substitute $u_h$ into the weak form (\ref{fem-form}) and set $v := \phi_i\;(i=1,\ldots,N)$, then we obtain a system of linear equations 
\begin{equation}
\label{laer}
A_hc_h  = b_h \, .
\end{equation}
Throughout the paper, we will use the matrices $S_h$, $E_h$ and $M_h$ with entries
\begin{equation}
\label{matrices}
	s_{ij} = \int_\Omega \left(G\, \nabla  \phi_j\cdot\nabla\phi_i\right)
	\text{,} \qquad e_{ij} = \int_\Omega \left(\textbf{w}\cdot\nabla\phi_j\right)\phi_i \text{,} \qquad  m_{ij} := \int_\Omega \phi_i\;\phi_j 
\end{equation}
($i,j=1,...,N$), \resp .
Then the matrix $A_h$ in system \eqref{laer} has a decomposed form
\begin{equation}
\label{decomporig}
A_h = {S_h} + {E_h}\, .
\end{equation}
The matrix $M_h$, called the mass matrix, will appear in representing the $L^2$-norm. 
In these matrices the index $h$ refers to the subspace $V_h$, where $h>0$ is the mesh parameter (typically the maximal diameter of the cells).

\subsection{The GMRES/GCR iterative methods and the convergence estimates} 
\label{s:cgn_est1}	
It is well-known how to approach the numerical solution of system \eqref{laer}, see, e.g. \cite{ESW, Saad}. This system is nonsymmetric, owing to the convection term, hence
we need to use nonsymmetric iterative methods.
There exist several nonsymmetric iterative methods,
of which the most commonly used ones minimize the residual error in the search subspace along the iteration. The most widespread one is GMRES, which gives equivalent errors to the simpler version GCR. 
Moreover, 
to achieve faster convergence, we apply preconditioning, where we choose    the symmetric part of $A_h$, i.e. $S_h := \frac{A_h+A_h^T}{2}$, as a preconditioner \cite{AKK14, Spillane}.
Thus we apply the iterative method to the preconditioned equation 
\begin{equation}
\label{preclaer}
S_h^{-1}A_hc = S_h^{-1}b\, ,
\end{equation}
in which we drop the index $h$ from $c$ and $b$ since indices for vectors will be used for iterates in what follows.
The arising preconditioned GCR iteration is given  below. It also involves
the energy norm of $S_h$, defined as $\|v\|_{S_h} := \sqrt{\innerproduct{{S_h}v}{v}}$.

\textbf{Algorithm.} The preconditioned GCR iteration  
 
\begin{algorithm}[H]
	\label{GCR-alg}
	\small
	$c_0 := \textbf{0}$\;
	
	\sm 
	
		$r_0 := {S_h}^{-1}A_hc_0 - {S_h}^{-1}b$; \qquad
	$p_0 := r_0$; \qquad
	$z_0 := {S_h}^{-1}A_hp_0$\;
	
	\sm 
	
	$k := 0$\;

	\sm 
	
	\While{$\|r_k\|_{S_h} > \operatorname{TOL}$}{
		
		\sm 
		
		$\alpha_k := -\dfrac{\innerproduct{r_k}{z_k}_{S_h}}{\|z_k\|_{S_h}^2}$\;
		\sm 
		
		$c_{k+1} := c_k + \alpha_k p_k$; \qquad
		$r_{k+1} := r_k + \alpha_k z_k$; \qquad
		$y_k := {S_h}^{-1}A_hr_{k+1}$\;
		
		\me
		
		\For{$i=0,1,\ldots,k$}{
			$\beta_{i,k} := -\dfrac{\innerproduct{y_k}{z_i}_{S_h}}{\|z_i\|_{S_h}^2}$\;
		}
		$p_{k+1} := r_{k+1} + \sum\limits_{i=0}^{k} \beta_{i,k}p_i$; \qquad
		$z_{k+1} := y_k + \sum\limits_{i=0}^{k} \beta_{i,k} z_i$\;
		$k := k + 1$\;
	}
\end{algorithm}

\subsubsection{Convergence of the iteration}
Now let us turn to the convergence of the method. First we note that the iteration finds the exact solution at the latest in the $N$th step. However, this theoretical property is not practical, since $N$ is very large (typically in the range $10^5$--$10^{10}$), whereas we obtain a good approximation in far fewer steps (typically fewer than 100).

There are well-known convergence estimates that involve the error through the residual 
error vectors,  defined  in the preconditioned case as 
$$
r_k:={S_h}^{-1}A_hc_k-{S_h}^{-1}b
$$ 
corresponding to system \eqref{preclaer}, 
 measured in the energy ${S_h}$-norm. 
These estimates are based on the above-mentioned residual minimization property 
\begin{equation}
\label{basic-ineq0}
\|r_k\|_{S_h} =
\min\limits_{P_k \in \pi_k^1}\;  \|P_k({S_h}^{-1}A_h) r_0\|_{S_h}
\text{,}
\end{equation}
where $\pi_k^1$ consists of polynomials of degree at most $k$ for which $P_k(0) = 1$, and $P_k(S_h^{-1}A_h)$ is the matrix obtained by evaluating $P_k$ at $S_h^{-1}A_h$.

When $S_h^{-1}A_h$ is $S_h$-normal, as is the case here, there is a basis of
$\mathbb{C}^N$ consisting of eigenvectors of $S_h^{-1}A_h$, and these can be
chosen to be orthonormal with respect to the $S_h$-inner product. Expanding
$r_0$ in this basis, \eqref{basic-ineq0} implies
\begin{equation}
\label{basic-ineqeig}
\|r_k\|_{S_h} \le \min_{P_k \in \pi_k^1} \max_{\lambda \in \sigma(S_h^{-1}A_h)}
|P_k(\lambda)| \, \|r_0\|_{S_h}.
\end{equation}

A Krylov iteration typically exhibits two phases: linear convergence followed by superlinear convergence. Here \emph{linear convergence} means that the average convergence factors are bounded by a constant less than 1, whereas
\emph{superlinear convergence} means that they decrease towards 0. That is,
the corresponding estimates are of the form
\begin{equation}
\label{linsuplin}
\Big(\frac{\|r_k\|_{S_h}}{\|r_0\|_{S_h}}\Big)^{1/k} \le Q < 1
\qquad\text{or}\qquad
\Big(\frac{\|r_k\|_{S_h}}{\|r_0\|_{S_h}}\Big)^{1/k} \le \ve_k
\qquad (k = 1,\dots,N),
\end{equation}
respectively, where $Q$ is a positive constant less than one, and $(\ve_k)$ is a sequence decreasing towards 0. Owing to the
special structure of the matrix $A_h$, superlinear convergence estimates turn
out to be more descriptive than linear ones.

 \subsubsection{Convergence estimation in terms of the eigenvalues of the perturbation matrix}
 
The superlinear convergence estimates depend on the singular values or eigenvalues of matrix ${S_h}^{-1}A_h$ in various forms and convey interesting information as briefly recalled below. 

In our case the basis is a matrix splitting, 
see \cite{eqopsurv, Moret}.
Namely, the matrix $A_h$ can be decomposed into its  symmetric part  ${S_h}$ 
and the antisymmetric part  ${E_h}$ as in \eqref{decomporig},   
where ${S_h}$  and ${E_h}$ have the entries given in 
 \eqref{matrices}.
Now \eqref{decomporig} yields
\begin{equation}
	\label{decomposition}
	{S_h}^{-1}A_h = I + {S_h}^{-1}{E_h}
	\text{,}
\end{equation}
which implies translated spectra:
\begin{equation}
	\label{translsp}
\lambda \in \sigma( {S_h}^{-1}A_h )    \qquad \iff \qquad 
\lambda = 1+\mu \quad \mbox{for some }
\mu \in  \sigma({S_h}^{-1}{E_h}) .
\end{equation}

Moreover, since ${E_h}$ is   antisymmetric,   ${S_h}^{-1}{E_h}$ is  ${S_h}$-antisymmetric (or ${S_h}$-skew-adjoint), i.e.   
$$
\innerproduct{{S_h}^{-1}{E_h}c}{d}_{S_h} = \innerproduct{{E_h}c}{d}= - \innerproduct{c}{{E_h}d}
= - \innerproduct{c}{{S_h}^{-1}{E_h}d}_{S_h}\, .
$$

Let $\mu_j:=\lambda_j({S_h}^{-1}{E_h})$ denote the eigenvalues of ${S_h}^{-1}{E_h}$ with non-increasing absolute values and also taking into account their multiplicity. Since $S_h^{-1}E_h$ is $S_h$-antisymmetric, its eigenvalues
$\mu_j = \lambda_j(S_h^{-1}E_h)$ are purely imaginary, and their moduli
coincide with the singular values $s_j(S_h^{-1}E_h)$ with respect to the
$S_h$-inner product. 
	In what follows, to simplify presentation, 
	we assume that ${E_h}$ is an invertible matrix, i.e., $\mu_j\neq 0$ for all $j$. (This holds, e.g., when 
	$\textbf{w}$ satisfies Definition 1 of global rectifiability from \cite{AKK14}, resulting in a corresponding invertible operator $E$, and when the mesh is fine enough such that ${E_h}$ inherits  invertibility.) We note, however, that even if  ${E_h}$ has zero eigenvalues for $j\ge N_0$, then $N_0$ is typically very large, and moreover,
	we might stop the iteration at step $N_0$ since we have reached the exact solution as shown below in \eqref{varsupbecs}. To summarize, the convergence problem is only interesting as long as $\mu_j\neq 0$.
 
 Altogether, based on the above, there exist real positive numbers  
\begin{equation}
\label{eig-notation}
\mbox{$\rho_1\ge \dots \ge \rho_N$ \ such that } \quad
s_j({S_h}^{-1}{E_h}) = |\lambda_j({S_h}^{-1}{E_h})| = |\mu_j|= \rho_j>0 \quad (\forall  j=1,\dots, N). 
\end{equation}
One can obtain various  convergence estimates for the preconditioned GMRES/GCR method in terms of the values $\rho_j$
from  \eqref{eig-notation}.
The most descriptive one for our study has 
 the following form: there exists a constant $K_0>0$ such that
\begin{equation}
\label{suplingen}
\frac{\|r_k\|_{S_h}}{\|r_0\|_{S_h}} \le \prod_{j=1}^k \frac{1}{\sqrt{1+  \frac{K_0}{\rho_j^2}}} \qquad (k=1,\dots, N) 
\text{.}
\end{equation}
Since $\rho_j$    decreases towards 0, so does the
convergence factor on the right side of \eqref{suplingen} as $k$ increases (see more insights below in Remark \ref{remopeig}).

We first note that a weaker bound, with an additional factor of $2$ in each term,
is easy to prove for $K_0 = 1$. Let
\[
\hat P_k(\lambda) := \prod_{j=1}^k
\frac{\lambda_j(S_h^{-1}A_h) - \lambda}{\lambda_j(S_h^{-1}A_h)} .
\]
Then $\hat P_k \in \pi_k^1$, and $\hat P_k(\lambda_i(S_h^{-1}A_h)) = 0$ for
$i = 1,\dots,k$. Substituting $\hat P_k$ into \eqref{basic-ineqeig}, applying
\eqref{translsp}, and using $|\mu_i| = \rho_i \le \rho_j$ for $i > k \ge j$, we
obtain
\[
\frac{\|r_k\|_{S_h}}{\|r_0\|_{S_h}}
\le \max_{i \ge k+1} |\hat P_k(\lambda_i(S_h^{-1}A_h))|
= \max_{i \ge k+1} \prod_{j=1}^k \frac{|\mu_j - \mu_i|}{|1+\mu_j|}
\le \prod_{j=1}^k \frac{2\rho_j}{\sqrt{1+\rho_j^2}}
= \prod_{j=1}^k \frac{2}{\sqrt{1 + \rho_j^{-2}}} .
\]
The similar estimation for $K_0=1$, without the factors of 2, is much more involved to derive, see \cite[Thm 1]{Moret}, but it is a 
descriptive fact that for any $k$ the expression  on the right side of \eqref{suplingen} has all terms less than 1, which is in accordance with the decrease of the residual norms.

We further note that various consequences of the above bound can be derived, also for similar estimations for standard CG methods for symmetric problems, see, e.g., \cite{Ax94, Nevanlinna, Winther}. Using \eqref{suplingen} with the above quoted value $K_0=1$ and applying well-known means, we have
\begin{equation}
	\label{varsupbecs}
\Big(	\frac{\|r_k\|_{S_h}}{\|r_0\|_{S_h}} \Big)^{1/k} \le
\Big(\prod_{j=1}^k \frac{\rho_j}{     \sqrt{1+ \rho_j^2}} \Big)^{1/k} \le \Big(\prod_{j=1}^k  \rho_j \Big)^{1/k} \le
\frac{1}{k} \sum\limits_{j=1}^k \rho_j
\le \frac{1}{k^{1/p}} \Big( \sum\limits_{j=1}^k \rho_j^p \Big)^{1/p}
\end{equation}
 for any fixed $p\ge 1$. In particular, it has been pointed out on the operator level \cite{Nevanlinna, Winther} that if, for some $p\ge 1$,
\begin{equation}
	\label{meascomp}
  \Lambda_p ({S_h}^{-1}{E_h}) := \sum\limits_{j=1}^\infty \rho_j^p \ < \infty,
\end{equation}
 then, defining $\Lambda := \Lambda_p ({S_h}^{-1}{E_h})^{1/p}$, we obtain
 $$
 \Big(	\frac{\|r_k\|_{S_h}}{\|r_0\|_{S_h}} \Big)^{1/k}  
 \le \frac{\Lambda}{k^{1/p}}  \qquad (k=1,\dots, N).
 $$
 This specifies the sequence 
 $\ve_k:=  \frac{\Lambda}{k^{1/p}}\to 0$ in  \eqref{linsuplin}, and shows that 
  the "$p$-degree of compactness" in \eqref{meascomp} leads to 
 a corresponding order $O(k^{-1/p})$ of the convergence estimate.
 
In the sequel we are interested in the  bound \eqref{suplingen} above.
 In order to have a 
 descriptive estimate, 
  sharper values for the constant $K_0>1$ will be discussed in 
paragraph \ref{subsubsharpK1}.

\begin{remark}
	\label{remopeig}
		The  differential operators  ${A},\; {S},\; {E}$, underlying the above matrices on the continuous PDE level  help us to understand the superlinear behavior. It can be shown (see, e.g.,  \cite{numsd}, Prop. 1.) that the operator ${S}^{-1}{E}: H^1_0(\Omega) \rightarrow H^1_0(\Omega)$ is compact.
		Therefore its eigenvalues  accumulate at zero, hence  
		the estimates on the right side of \eqref{varsupbecs} on the  operator level are sequences that tend to 0. On the matrix level the  values $\rho_j=|\lambda_j({S_h}^{-1}{E_h})|$ decrease close to zero (even closer if the dimension of $V_h$ is large), thus
		the  estimates on the right side of \eqref{varsupbecs} show "numerical convergence" to 0. 
\end{remark}

\subsection{Relationship between the convergence factors and the Laplace eigenvalues}\label{rellapeig}

The singular values $\rho_j = s_j({S_h}^{-1}{E_h})$ in  \eqref{eig-notation} can be further estimated using an upper bound from \cite{AxK}. In the referred article at (13), it is established that there exists a constant $K_1 > 0$  such that
\[
	\|{{S}}^{-1}{E}u\|_{{S}} \leq K_1 \|u\|_{L^2} \qquad (\forall u \in H^1_0(\Omega))
	\text{,}
\]
where $K_1 \leq \dfrac{\|\textbf{w}\|_{L^{\infty}}}{\sqrt{m}}$ with $m$ being the lower bound of $G$ in \eqref{eq:nonsymG_below}.
This can be used similarly as in the proof of Prop. 4.4 in \cite{AxK} to show the desired inequality on the operator level:
\begin{equation}
	\label{S_eigs}
	s_j^2({S}^{-1}{E}) \leq \frac{K_1^2}{\lambda_j({S})}
	\text{,}
\end{equation}
where $\lambda_j({S})$ denote the eigenvalues of operator ${S}$ in  non-decreasing  order. By the definition of matrices ${S_h}$ and ${E_h}$, this inequality is obtained also on the discrete matrix level if we replace the Sobolev space by the FE subspaces:

\begin{proposition}
    Let $M_h$ denote the mass matrix whose elements are defined in \eqref{matrices}.   Then 
    \begin{equation}
    	\label{sing}
    	\rho_j^2 = s_j^2({S_h}^{-1}{E_h}) \leq \frac{   K_1^2   }{\lambda_j(M_h^{-1}{S_h})}
    	\text{.}
    \end{equation}
\end{proposition}
\begin{proof*}
    Let us express the squares of the singular values by their variational characterization:
    \begin{equation*}
        s_j^2({S_h}^{-1}{E_h})
        =
        \min_{H_{j-1}\,\subset\,\mathbb{R}^N}
        \max_{\substack{\textbf{c}\,\perp_{S_h}\, H_{j-1} \\ \textbf{c}\neq \textbf{0}}}
        \frac{\innerproduct{(S_h^{-1}E_h)^*S_h^{-1}E_h\textbf{c}}{\textbf{c}}_{S_h}}{\|\textbf{c}\|_{S_h}^2}
        =
        \min_{H_{j-1}\,\subset\,\mathbb{R}^N}
        \max_{\substack{\textbf{c}\,\perp_{S_h}\, H_{j-1} \\ \textbf{c}\neq \textbf{0}}}
        \frac{\|S_h^{-1}E_h\textbf{c}\|_{S_h}^2}{\|\textbf{c}\|_{S_h}^2}
    \end{equation*}
    where $H_{j-1}$ stands for an arbitrary $(j-1)$-dimensional subspace. The norm in the numerator is bounded as follows:
    \begin{align*}
        \|S_h^{-1}E_h\textbf{c}\|_{S_h}
        &=
    \sup_{\substack{\textbf{d}\in\mathbb{R}^N \\ \|\textbf{d}\|_{S_h}=1}}
        |\innerproduct{E_h\textbf{c}}{\textbf{d}}|
        =
        \sup_{\substack{v_h\in V_h \\ \|v_h\|_{S}=1}}
        \left|\int_\Omega (\textbf{w}\cdot\nabla u_h)v_h\right|
        =  \sup_{\substack{v_h\in V_h \\ \|v_h\|_{S}=1}}
        |\la  Eu_h, v_h \ra_{L^2} | 
        \\ &
 =  \sup_{\substack{v_h\in V_h \\ \|v_h\|_{S}=1}}
        |\la S^{-1} Eu_h, v_h \ra_S | \le 
\sup_{\substack{v \in H^1_0(\Omega) \\ \|v\|_{S}=1}}
         |\la S^{-1} Eu_h, v  \ra_S |
        = \|{{S}}^{-1}{E}u_h\|_{{S}} \leq K_1 \|u_h\|_{L^2} 
        =
        K_1 \|\textbf{c}\|_{M_h}
    \end{align*}
     where $u_h$ and $v_h$ are the functions in $V_h$ determined by the coefficient vectors $\textbf{c}$ and $\textbf{d}$ according to \eqref{fem-coeff}.
    Using the derived estimate and that $\|\textbf{c}\|_{M_h}^2
        =
        \innerproduct{M_h\textbf{c}}{\textbf{c}}
        =
        \innerproduct{S_h^{-1}M_h\textbf{c}}{\textbf{c}}_{S_h}$, the squares of the singular values are bounded as
    \begin{equation*}
       \pushQED{\qed}
       s_j^2({S_h}^{-1}{E_h})
        \leq
        \min_{H_{j-1}\,\subset\,\mathbb{R}^N}
        \max_{\substack{\textbf{c}\,\perp_{S_h}\, H_{j-1} \\ \textbf{c}\neq \textbf{0}}}
        \frac{K_1^2\|\textbf{c}\|_{M_h}^2}{\|\textbf{c}\|_{S_h}^2}
        =
        K_1^2 \lambda_j(S_h^{-1}M_h)
        =
        \frac{K_1^2}{\lambda_j(M_h^{-1}S_h)} \text{.}
        \qedhere
		\popQED
    \end{equation*}
\end{proof*}

\sm 

Since $x\mapsto \frac{x}{\sqrt{1+x^2}}$ is strictly  increasing on $\er^+$,
 using \eqref{sing}, we readily obtain from (\ref{suplingen})  the following upper bound:
\begin{equation}
	\label{relate}
\frac{\|r_k\|_{S_h}}{\|r_0\|_{S_h}} \leq
       \prod_{j=1}^k \frac{1}{ \sqrt{1 +  \, \tilde K \  \lambda_j(M_h^{-1}{S_h}) } } \, ,
\end{equation}
where 
\begin{equation}
\label{tildeK}
\tilde K:= \frac{K_0}{   K_1^2  }\, .
\end{equation}
This can be further estimated using the inequality
\begin{equation}
	\label{third-ineq}
	\lambda_j(M_h^{-1}{S_h}) \geq \lambda_j({S})\quad(\forall j =1,\ldots,N)\text{,}
\end{equation}
which follows from the min-max description of eigenvalues. More details can be found in \cite[Section 8]{boffi}. Hence, the final estimate is
\begin{equation}
	\frac{\|r_k\|_{S_h}}{\|r_0\|_{S_h}} \leq
 \prod_{j=1}^k \frac{1}{ \sqrt{1 +  \, \tilde K \  \lambda_j({S}) } }	
	\text{.}
\end{equation}
That is, altogether, the estimation is recast to a sequence determined by the eigenvalues of operator ${S}$.

\subsection{Eigenvalues of the continuous and discrete Laplacian}
\label{s:cgn_est2}

As seen in the previous subsection, the convergence estimates for the
discretized convection-diffusion problem are ultimately expressed in terms of
the eigenvalues of the operator $S$, where $S = -\Delta$ in the main case
\eqref{nonsym} and $Su = -\divergence(G\,\nabla u)$ in general. These
eigenvalues and their discrete counterparts depend on the domain
$\Omega \subset \R^2$, and this dependence is the focus of this paper. We
therefore use the following notation.

The eigenvalues of the (Dirichlet) Laplacian on $\Omega$ are denoted by
\begin{equation}
\label{conteigsq}
0 < \lambda_1(\Omega) \le \lambda_2(\Omega) \le \dots \to +\infty;
\end{equation}
their basic properties are summarized in Section~\ref{s:continuity_geometry}.
The eigenvalues of the discrete Laplacian, which already appeared in
\eqref{sing}, are denoted by
\begin{equation}
\label{disceigsq}
0 < \lambda^h_1(\Omega) \le \lambda^h_2(\Omega) \le \dots \le \lambda^h_N(\Omega),
\end{equation}
and are defined by the FE eigenproblem
\begin{equation}
\label{weakeigp}
\int_\Omega \nabla u^h_j \cdot \nabla v_h
= \lambda_j^h(\Omega) \int_\Omega u^h_j\, v_h
\qquad (\forall v_h \in V_h).
\end{equation}
Writing $u^h_j = \sum_{\ell=1}^N c_{j,\ell}\,\phi_\ell$ and denoting the
stiffness and mass matrices \eqref{matrices} by
\begin{equation}
\label{shmhomx}
S_{h,\Omega} = \{s_{ij}\}_{i,j=1}^N, \qquad M_{h,\Omega} = \{m_{ij}\}_{i,j=1}^N
\end{equation}
to indicate their dependence on $\Omega$, problem \eqref{weakeigp} becomes
$S_{h,\Omega}\, c_j = \lambda_j^h(\Omega)\, M_{h,\Omega}\, c_j$. Hence
\begin{equation}
\label{disceigsqmtx}
\lambda_j^h(\Omega) = \lambda_j\bigl(M_{h,\Omega}^{-1} S_{h,\Omega}\bigr).
\end{equation}

The same notation is used for the general operator
$Su = -\divergence(G\,\nabla u)$, with $\nabla u^h_j \cdot \nabla v_h$ in
\eqref{weakeigp} and $\nabla\phi_i \cdot \nabla\phi_j$ in \eqref{matrices}
replaced by $G\,\nabla u^h_j \cdot \nabla v_h$ and
$G\,\nabla\phi_i \cdot \nabla\phi_j$, respectively.

\subsection{Sharpness of the convergence estimates}
\label{s:cgn_est2_2}
In the previous subsection, we have established a connection between the ${S_{h, \Omega}}$-norm of the residuals and the eigenvalues of operator ${{S}}$, using the following chain of inequalities:
\begin{align}
\label{ineq-chain1}
\frac{\|r_k\|_{S_{h, \Omega}}}{\|r_0\|_{S_{h, \Omega}}} & \leq \prod_{j=1}^k \frac{1}{\sqrt{1+  \ltfrac{K_0}{  
			 s_j ^2({S_{h, \Omega}\iv}\, {E_{h, \Omega}}) }
			}
	     }
\\
\label{ineq-chain2}
& \leq \prod_{j=1}^k \frac{1}{ \sqrt{1 +  \, \tilde K \  \lambda_j^h(\Omega)  } } 
\\
\label{ineq-chain3}
& \leq 
\prod_{j=1}^k \frac{1}{ \sqrt{1 +  \, \tilde K \  \lambda_j({\Omega}) } }	
\text{,}
\end{align}
where we recall that \eqref{ineq-chain1} follows from \eqref{suplingen} and the notation \eqref{eig-notation}, further, \eqref{ineq-chain2}  follows from \eqref{sing} and the notations \eqref{eig-notation} and \eqref{tildeK},
and finally, \eqref{ineq-chain3}  follows from \eqref{third-ineq}
 and the notations \eqref{conteigsq} and \eqref{disceigsqmtx}.

The aim of this paper is to show that the continuous dependence of the eigenvalues $\lambda_j({\Omega})$, appearing in \eqref{ineq-chain3}, on $\Omega$, is reflected in the 
continuous dependence of the sequence of residual norms, appearing on the left side of \eqref{ineq-chain1}, on $\Omega$. 
For this, it is important to see how descriptive this estimation is.  
Therefore, in this section  we examine the sharpness of the above three intermediate estimates   with  analytical results and numerical experiments, tracing back the inequalities from \eqref{ineq-chain3} to \eqref{ineq-chain1}.

\subsubsection{\texorpdfstring{Sharpness of the estimate \eqref{ineq-chain3}}{Sharpness of the estimate}}
\label{subsubssharp1}

As seen above, using the notations \eqref{conteigsq} and \eqref{disceigsqmtx}, the estimate \eqref{ineq-chain3}  follows directly from inequality \eqref{third-ineq}, which states that
\begin{equation}
\label{third-ineq28}
\lambda_j^h(\Omega) \geq \lambda_j(\Omega)  
\end{equation}
for all $j =1,\ldots,N$. This inequality is sharp in the sense that there is convergence when $h\to 0$, as stated below.

\begin{proposition}[{\cite[Section~8]{boffi}}]
		\label{propbof1} 
 If $\Omega$ is a polygonal domain and $\{ V_h\}_{h>0}$ is a family of piecewise linear finite element subspaces in $\Omega$ where the mesh parameter $h\to 0$, then for all fixed $j \in \mathbb{N}$, we have
 	\[
 	\lambda_j^h(\Omega) \to \lambda_j(\Omega) \qquad \mbox{as} \ h\to 0.
 	\]
 \end{proposition}

Figure \ref{num-test-3}  illustrates the sharpness on the 100 smallest eigenvalues. 
We plot the eigenvalues of the operator $S$ (black) and its discrete counterpart (red) on a uniform mesh with 300 nodes in each direction  of the unit square.
On the left, we consider the Laplacian
$$
Su=- \Delta u,
$$
and on the right, the more general operator
\begin{equation}
\label{gengve}
 {S}u= - \divergence(G\, \nabla  u)  \qquad \mbox{with } \quad G := \begin{pmatrix}
\ve & 0\\
0 & 1
\end{pmatrix} \, ,
\end{equation}
where $\ve=0.01$. We can observe in both cases that the operator eigenvalues $\lambda_j(\Omega) $ are indeed always below the matrix eigenvalues $\lambda_j^h(\Omega) $ and the two curves are almost identical.

\begin{figure}[h!]
	\centering
	\includegraphics[width=1\textwidth]{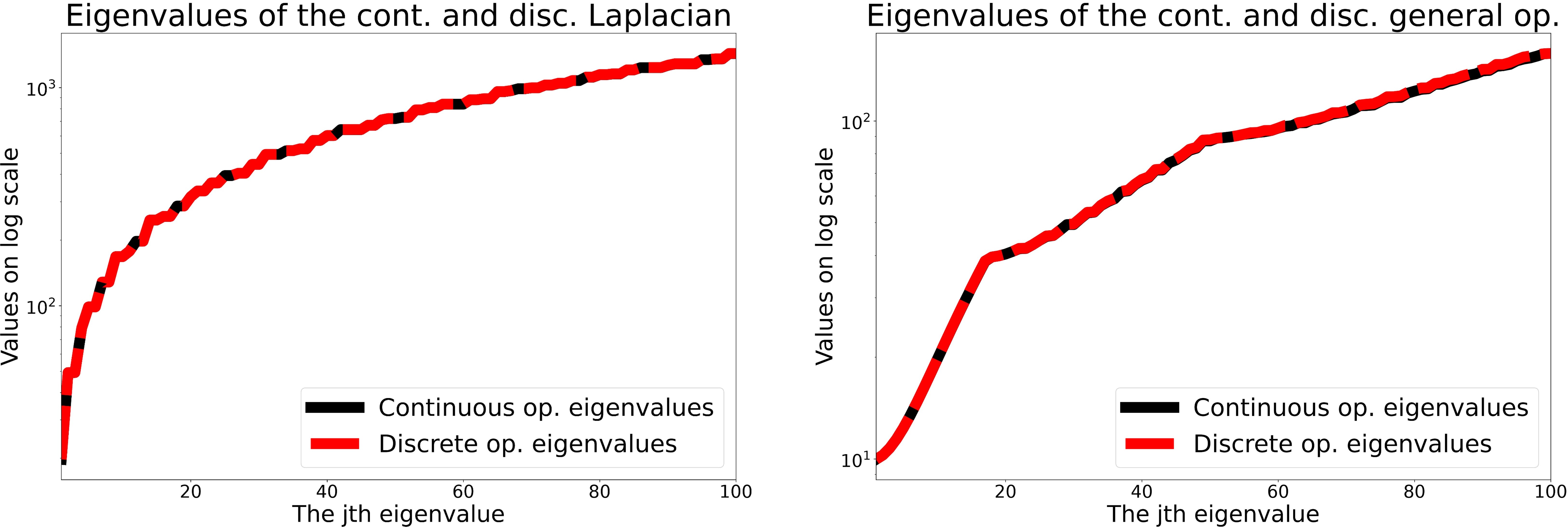}
	\caption{The 100 smallest eigenvalues of operator ${S}$ and the matrix $M_{h, \Omega}^{-1}{S_{h, \Omega}}$ in the numerical tests.}
	\label{num-test-3}
\end{figure}

\subsubsection{Sharpness of the estimate \texorpdfstring{\eqref{ineq-chain2}}{(ineq-chain2)}}
\label{sharp-ii}

The estimate \eqref{ineq-chain2} is a consequence of the inequality  \eqref{sing}  
that relates the singular values of ${S_{h, \Omega}\iv}\, {E_{h, \Omega}}$ to the eigenvalues $\lambda_j^h(\Omega)$ (with the notation \eqref{disceigsqmtx}). 
We first discuss the sharpness of the similar inequality on the operator level, which then indicates sharpness  on the discrete level in the sense of the approximation property used in subsection  \ref{subsubssharp1}.

On the operator level, the corresponding inequality is  \eqref{S_eigs}:
\begin{equation}
\label{S_eigs2}
s_j^2({S}^{-1}{E}) \leq \frac{K_1^2}{\lambda_j({S})}
\text{.}
\end{equation}
We show in an example that this bound can be optimal. Let us consider the operator on the left side of \eqref{eq:nonsymG_op} on the unit square $\Omega = (0,1)^2$, where $S$ is from \eqref{gengve} and $\textbf{w}$ is parallel to the $x$-axis, that is, with the coefficients 
\begin{equation}
\label{coeff27}
G := \begin{pmatrix}
\ve & 0\\
0 & 1
\end{pmatrix} \, , \qquad \textbf{w} :=  (\beta, 0)
\end{equation}
where $\epsilon > 0$ and $\beta > 0$ are constants.  In this  case  we can calculate analytically the values in \eqref{S_eigs2} with separation of variables, given conveniently with double-index notation: 
\begin{equation}
s_{k,l}^2({S}^{-1}{E})  = \dfrac{ \beta^2}{4\pi^2 \epsilon (   k^2 +  \epsilon  l^2)}
\text{,} \qquad
\lambda_{k,l}({S})  = \pi^2 (  k^2 +  \epsilon  l^2)\text{,}
\end{equation}
where double indexing $k,l \in \mathbb{N}^+$ is used.
(This covers the Laplacian principal part as a special case with $\epsilon=1$.)
We observe that
	\begin{equation}
	\label{sing-est}
	s_{k,l}^2(S^{-1}E) = \frac{\kappa^2}{\lambda_{k,l}(S)}
	\qquad \text{with} \quad \kappa^2 = \frac{\beta^2}{4\epsilon} ,  
	\end{equation}
thus these values as subsets of $\er$ are directly proportional.
	Taking also into account that each $s_{k,l}$ has
	multiplicity two, listing the singular values and eigenvalues with multiplicity
	in single-index form gives
	\[
	s_{2m-1}^2(S^{-1}E) = s_{2m}^2(S^{-1}E) = \frac{\kappa^2}{\lambda_m(S)}
	\qquad (m \in \mathbb{N}^+).
	\]
	By Weyl's law, $\lambda_{2m}/\lambda_m \to 2$, {hence we obtain asymptotic proportionality:
	\[
	\lim_{j\to\infty} s_j^2(S^{-1}E)\,\lambda_j(S) = 2\kappa^2 .  
	\]

On the discrete matrix level,  the  sharpness of the inequality  \eqref{sing}
is understood in the sense that one can get close to equality if the mesh is sufficiently fine. This follows in the above example because the discrete singular values converge to those of the corresponding operators when the mesh parameter $h\to 0$.  
To illustrate this, we present   two numerical tests with the coefficients \eqref{coeff27} on a uniform mesh with 300 nodes in each direction. Firstly, let us consider $\epsilon=\beta=1$ (i.e. $S=-\Delta$), and secondly, let  $\epsilon = 0.01$ as in \eqref{gengve} and let $\beta=1$. The 100 largest numerically calculated   values  of $ s_{k,l}^2({S_{h, \Omega}}^{-1}{E_{h, \Omega}})$ and 
 $\kappa^2 /  \lambda_{k,l}^h(\Omega)$  (that is, the discrete analogues of the two sides of   (\ref{sing-est})) can be seen in Figure \ref{num-test-2}, 
showing the values are indeed close to equal.

\begin{figure}[h!]
	\centering
	\includegraphics[width=1\textwidth]{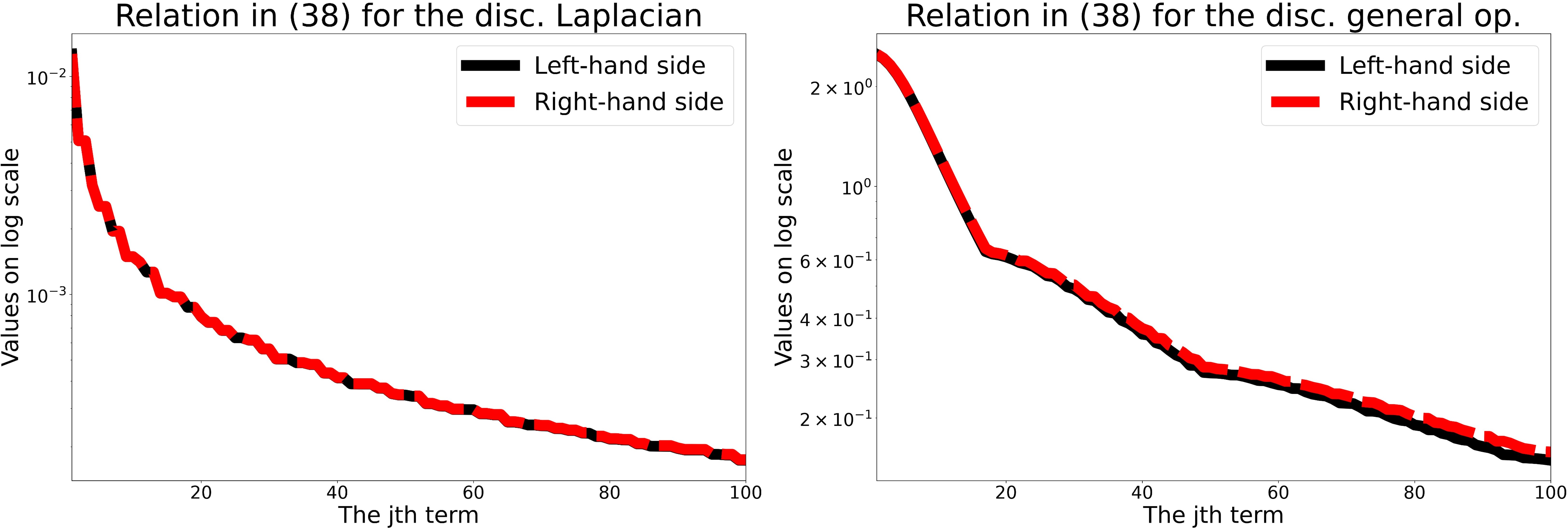}
	\caption{Numerically calculated   values  of $ s_{j}^2({S_{h, \Omega}}^{-1}{E_{h, \Omega}})$ and $\kappa^2 /  \lambda_{j}^h(\Omega)$ 		
		  on a uniform mesh.   }
	\label{num-test-2}
\end{figure}

\subsubsection{Sharpness of the estimate \texorpdfstring{\eqref{ineq-chain1}}{(ineq-chain1)}}
\label{subsubsharpK1}

This inequality is the superlinear convergence estimate (\ref{suplingen}):
\begin{equation}
\label{suplingen2}
\dfrac{\|r_k\|_{S_{h, \Omega}}}{\|r_0\|_{S_{h, \Omega}}}  \le \prod_{j=1}^k \frac{1}{\sqrt{1+  \frac{K_0}{\rho_j^2}}} \qquad (k=1,\dots, N) 
\text{,}
\end{equation}
wherein \, 
$\rho_j:=s_j({S_{h, \Omega}\iv}\, {E_{h, \Omega}})$. 
As discussed in subsection \ref{s:cgn_est1}, we know $K_0\ge 1$ from \cite{Moret} but the sharp value of $K_0$ is yet unknown. 
By extensive numerical experimentation, we found that a sharp superlinear  estimate might hold with the value 
$$
K_0=2 .
$$ 
Namely, the right side of (\ref{suplingen2}) with $K_0=2$ gave an upper bound for all runs, whereas for some of them it produced the true convergence curve with numerically negligible difference. 
To illustrate the latter, we present here two  tests. Let $\Omega_1 := (0,0.75)\times(0,2)$ and let $\Omega_2$ be the inscribed square of the unit circle with diagonals coinciding with the axes. We discretize these domains uniformly with a grid consisting of 780 and 1054 nodes and we use first order Courant elements. We set the data of (\ref{coeff27}) the following way:
\begin{itemize}
	\item 
	The coefficients \eqref{coeff27} are used.  
	Firstly, we consider $\epsilon=1$ (Laplacian) and $\beta=2$ on $\Omega =\Omega_1$, 
	and secondly, let  $\epsilon = 0.001$ and $\beta=0.2$ on $\Omega = \Omega_2$. 
\item The vector $b_h$ from (\ref{laer}) is the sum of the real and imaginary parts of eigenvectors corresponding to the 100 eigenvalues of matrix ${S_{h, \Omega}}^{-1}{E_{h, \Omega}}$ with largest magnitude.
\end{itemize}
The result of the numerical tests with these two sets of coefficients can be seen in Figure \ref{num-test-1}. 
The relative residual norms  and the superlinear estimate 
$$
\dfrac{\|r_k\|_{S_{h, \Omega}}}{\|r_0\|_{S_{h, \Omega}}}
\qquad \mbox{and}
\qquad 
\prod_{j=1}^k \frac{1}{\sqrt{1+  \frac{2}{\rho_j^2}}},
$$
\resp \ (where the latter is the right side of (\ref{suplingen2}) with $K_0=2$) are denoted with a black solid line and a  red dashed line, \resp . We observe that the estimations closely follow the real convergence curves. 

Proving sharpness analytically for $K_0=2$ is the subject of future work at the moment. A partial theoretical  result where the value $K_0=2$ appears is given below for initial vectors restricted to a proper subspace. For this, recall from subsection~\ref{s:cgn_est1} that the $N \times N$ matrix
$S_{h,\Omega}^{-1}E_{h,\Omega}$ is $S_{h,\Omega}$-antisymmetric. Let $N = 2n$
be even and assume that $S_{h,\Omega}^{-1}E_{h,\Omega}$ is nonsingular. Then its
eigenvalues are $\pm i\tau_m$ $(m = 1,\dots,n)$ with
$\tau_1 \ge \tau_2 \ge \dots \ge \tau_n > 0$, hence the eigenvalues of
$S_{h,\Omega}^{-1}A_{h,\Omega}$ are $1 \pm i\tau_m$. Since
$S_{h,\Omega}^{-1}E_{h,\Omega}$ is $S_{h,\Omega}$-normal, its singular values
are the moduli of its eigenvalues, so
\[
\rho_{2m-1} = \rho_{2m} = \tau_m \qquad (m = 1,\dots,n).
\]
In particular, for $k = 2p$ the right-hand side of \eqref{suplingen2} with
$K_0 = 2$ equals
\[
\prod_{j=1}^{k} \frac{1}{\sqrt{1 + 2/\rho_j^2}}
= \prod_{j=1}^{k} \frac{\rho_j}{\sqrt{2 + \rho_j^2}}
= \prod_{m=1}^{p} \frac{\tau_m^2}{2 + \tau_m^2}.
\]

\begin{proposition}
\label{prop:K0is2}
Let $N = 2n$, let $S_{h,\Omega}^{-1}E_{h,\Omega}$ be nonsingular, and let
$k = 2p$ with $1 \le p \le n$. Assume that the $S_{h,\Omega}$-orthonormal
eigenvector expansion of $r_0$ has nonzero components only for eigenvalues
$1 \pm i\tau_\ell$ with $\tau_\ell \le \tau_p$. Then $r_k$ satisfies
\eqref{suplingen2} with $K_0 = 2$:
\[
\frac{\|r_k\|_{S_{h,\Omega}}}{\|r_0\|_{S_{h,\Omega}}}
\le \prod_{j=1}^{k} \frac{1}{\sqrt{1 + 2/\rho_j^2}} .
\]
\end{proposition}

\begin{proof*}
Let $z_m := 1 + i\tau_m/\sqrt2$ and define the polynomial
\[
\tilde P_k(\lambda) := \prod_{m=1}^{p}
\Big(1 - \frac{\lambda}{z_m}\Big)\Big(1 - \frac{\lambda}{\overline{z_m}}\Big),
\]
which has degree $k=2p$, real coefficients, and satisfies $\tilde P_k(0) = 1$.
Let $\Lambda$ be the set of eigenvalues $1 \pm i\tau_\ell$ with
$\tau_\ell \le \tau_p$. Using \eqref{basic-ineq0} and expanding $r_0$ in the
$S_{h,\Omega}$-orthonormal eigenbasis as in the derivation of
\eqref{basic-ineqeig}, the assumption on $r_0$ gives
\begin{equation}
\label{maxPlamd}
\frac{\|r_k\|_{S_{h,\Omega}}}{\|r_0\|_{S_{h,\Omega}}}
\le \frac{\|\tilde P_k(S_{h,\Omega}^{-1}A_{h,\Omega})\, r_0\|_{S_{h,\Omega}}}
         {\|r_0\|_{S_{h,\Omega}}}
\le \max_{\lambda \in \Lambda} |\tilde P_k(\lambda)| .
\end{equation}
Let $\lambda = 1 + i\sigma \in \Lambda$, so that $\sigma^2 \le \tau_p^2$. For
$m = 1,\dots,p$ we have $z_m - \lambda = i(\tau_m/\sqrt2 - \sigma)$ and
$\overline{z_m} - \lambda = -i(\tau_m/\sqrt2 + \sigma)$, hence
\[
\Big|\Big(1 - \frac{\lambda}{z_m}\Big)\Big(1 - \frac{\lambda}{\overline{z_m}}\Big)\Big|
= \frac{|\tau_m^2/2 - \sigma^2|}{1 + \tau_m^2/2}
= \frac{|\tau_m^2 - 2\sigma^2|}{2 + \tau_m^2}
\le \frac{\tau_m^2}{2 + \tau_m^2},
\]
where the last step uses $0 \le \sigma^2 \le \tau_p^2 \le \tau_m^2$, so that
$-\tau_m^2 \le \tau_m^2 - 2\sigma^2 \le \tau_m^2$. Taking the product over $m$
and using the identity above, we obtain
$|\tilde P_k(\lambda)| \le \prod_{j=1}^{k} \big(1 + 2/\rho_j^2\big)^{-1/2}$
for all $\lambda \in \Lambda$. Substituting this into \eqref{maxPlamd}
completes the proof.
\qed
\end{proof*}

\begin{figure}[h!]
	\centering
	\includegraphics[width=1\textwidth]{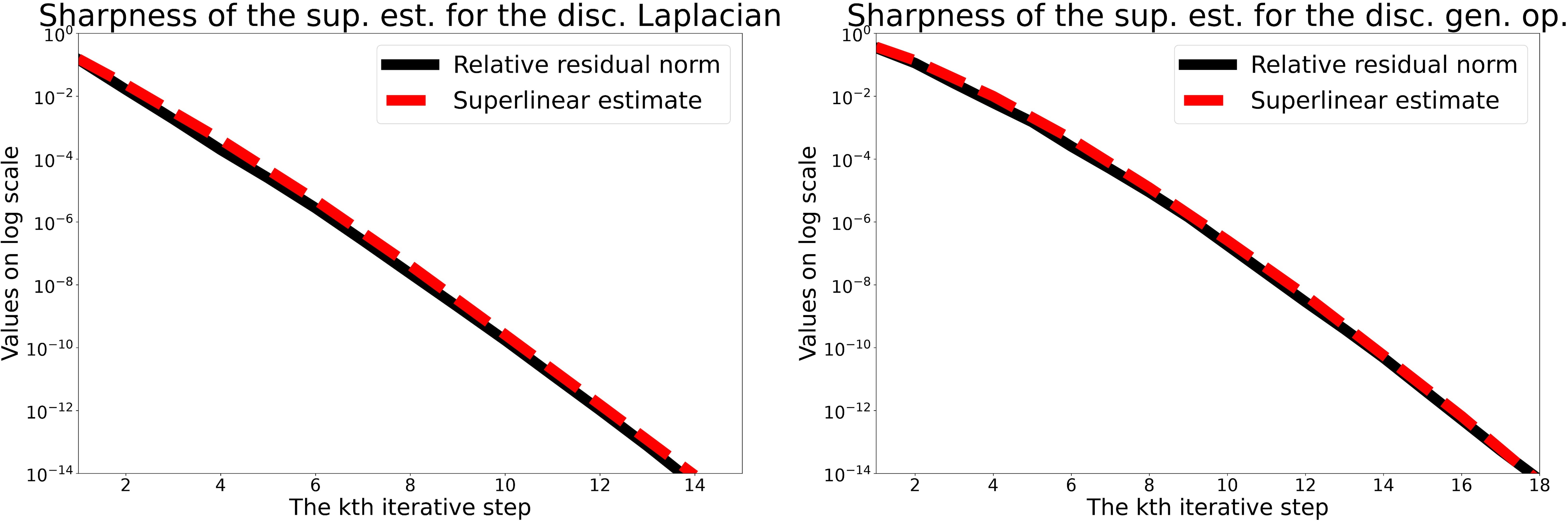}
	\caption{Numerical tests to support the claim that the superlinear estimate (\ref{suplingen2}) may be sharp for $K_0=2$.}
	\label{num-test-1}
\end{figure}

\subsubsection{Conclusion} 

In this section we have seen how the behaviour of the GMRES/GCR iteration is described by proper eigenvalues. In the previous three subsections, we studied the relevance of the three inequalities in  \eqref{ineq-chain1}--\eqref{ineq-chain3}  using analytical and numerical tools. 
We have observed that these inequalities are sharp in a proper sense either  in general or for suitable model problems. This indicates the well-founded expectation that  the qualitative properties of the sequence of final bounds
\begin{equation}
	\label{op-est}
	q_k(\Omega) :=
	\prod_{j=1}^k
	\frac{1}{ \sqrt{1 +  \, \tilde K \  \lambda_j({\Omega}) } }
	\qquad (k\in \mathbb{N}^+)
\end{equation}
in \eqref{ineq-chain3} are inherited also by the sequence of relative residual errors that occur in the iterative solution of the convection-diffusion problem.

In particular, we can gain an insight into the continuous dependence of the 
behaviour of the iteration on the domain $\Omega$ if we study the continuous dependence of the value of $q_k(\Omega)$ on  the geometry of $\Omega$. 
These investigations come in the next section. Therein one studies families of domains $\Omega$ contained in a fixed compact box domain $B$, hence the value  $\tilde K$ can be fixed uniformly. Altogether, we are led to study the continuous dependence of the eigenvalues $\lambda_j({\Omega})$  on the geometry of the domain $\Omega$.

\section{Continuous domain dependence of Laplace eigenvalues under geometric deformations} \label{s:continuity_geometry}

 In this section we study how the eigenvalues of the Dirichlet Laplacian depend on the underlying domain, and we extend the analysis to the uniformly elliptic operators in divergence
form introduced in \eqref{eq:nonsymG_op}. The study of the spectral dependence of elliptic differential operators on geometric perturbations of the underlying domain is a classical topic in mathematical analysis; see, e.g., \cite{henrot}, \cite{buttazzo}, \cite{mosco}. Our
aim here is to isolate sufficient geometric conditions for continuity and to deduce from them the continuity of the function $q_k(\Omega)$ defined in \eqref{op-est}.  Since our approach is based on abstract functional analytic techniques, the results apply to domains $\Omega$ in $n$ dimensions, for any $n \geq 2$, with the exception of those results given specifically for planar domains. 

\subsection{The continuity of eigenvalues}
Throughout this section we call a bounded set of the form 
\[B=(0,a_1)\times\dots\times(0,a_n)\subset \mathbb{R}^n, \qquad 0<a_1,\dots,a_n<\infty,\] 
an open box domain. We call its closure $\B$ a compact box domain. Given an open set $\Omega\subset \overline{B}$, we identify each $u$ in the Sobolev space $H_0^1(\Omega)$ with its extension by zero, which belongs to $H_0^1(B)$. In this way, the notions of distances and convergence in $L^2(B)$ make sense for functions defined on varying domains.

For an open set $\Omega \subset \B$, we  define the Dirichlet Laplace operator to be the operator $-\Delta$ associated to the bilinear form
\begin{equation}
    \label{eq:for}
    a_\Omega:H_0^1(\Omega)\times H_0^1(\Omega)\rightarrow \mathbb{R},\qquad a_\Omega(u,v)=\int_\Omega\nabla u\cdot\nabla v. 
\end{equation}
On smooth functions, the Laplace operator on $\R^n$ is simply 
$-\Delta =-\sum_{j=1} ^n \frac{\partial^2}{\partial x_j ^2}.$
We define the operator domain
\begin{equation}
    \mathcal{D}(\Omega) = \{ u \in H_0^1(\Omega) : -\Delta u \in L^2(\Omega) \}
    \label{op: domain}
\end{equation}  
where $-\Delta$ is understood in the distributional sense. The differential expression is the same for every $\Omega$ and what distinguishes one underlying geometric domain in $\R^n$ from another is the operator domain \eqref{op: domain}. This allows us to encode boundary conditions with no regularity assumptions on $\partial\Omega$. In this setting, $-\Delta:\mathcal{D}(\Omega)\subset H_0^1(\Omega)\subset L^2(\Omega)\rightarrow L^2(\Omega)$ is a self-adjoint, positive operator.
The compactness of the embedding $H_0^1(\Omega)\hookrightarrow L^2(\Omega)$ implies the compactness in $L^2(\Omega)$ of the resolvent of the Dirichlet Laplacian. We define the eigenvalue problem
\[-\Delta u=\lambda u, \qquad u\in\mathcal{D}(\Omega).\]
We say that $u\in H_0^1(\Omega)$, $u\neq 0$ is a weak eigenfunction if it is a solution of the variational problem
\[\int_\Omega\nabla u\cdot\nabla v =\lambda\int_\Omega uv, \qquad  \forall v\in H_0^1(\Omega).\]
By the spectral theorem for compact, self-adjoint operators applied to the resolvent, the spectrum of the Laplacian is therefore purely discrete and consists of a sequence of eigenvalues of finite multiplicity
\[0<\lambda_1(\Omega)\le\lambda_2(\Omega)\leq\dots\uparrow\infty,\]
repeated according to multiplicity, and $L^2(\Omega)$ admits an orthonormal
basis of eigenfunctions. Equivalently, by 
the Courant--Fischer min--max
principle\begin{equation}\label{eq:minmax}
\lambda_j(\Omega)=\min_{\substack{V\subset H^1_0(\Omega)\\ \dim V=j}}\
 \max_{u\in V\setminus\{0\}}\
\frac{\displaystyle\int_\Omega|\nabla u|^2}{\displaystyle\int_\Omega u^2}.
\end{equation}
Here, $V\subset H_0^1(\Omega)$ is a linear subspace. Formula \eqref{eq:minmax} depends on $\Omega$ only
through the space $H^1_0(\Omega)\subset H^1_0(B)$, immediately yielding the following property.
\begin{remark}[Domain monotonicity]
\label{rem:monotonicity}
If $\Omega'\subset\Omega\subset B$, then $H^1_0(\Omega')\subset H^1_0(\Omega)$ (extend by zero),
so the minimum in \eqref{eq:minmax} for $\Omega$ ranges over a larger family of subspaces than
for $\Omega'$, giving
\[
  \Omega'\subset\Omega \quad\Longrightarrow\quad \lambda_j(\Omega)\le\lambda_j(\Omega') \qquad \forall j\in\mathbb{N}.
\]
Thus a smaller domain has larger eigenvalues, and a larger domain has smaller ones.
\end{remark}

A first conceptual difficulty lies in the choice of an appropriate notion of domain convergence. Depending on the context, several topologies or notions of deformation may be considered: convergence in measure, Hausdorff convergence of the complements, deformations via diffeomorphisms, or Mosco convergence of the associated function spaces. A natural approach is provided by the notion of $\gamma$-convergence; see \cite{buttazzo}.

\begin{definition}[$\gamma$-convergence]
Let $\{\Omega_k\}_{k\in\mathbb{N}}$ be a sequence of open sets of $B$, and let $\Omega \subset B$ be an open set. We say that $\Omega_k$ \emph{$\gamma$-converges} to $\Omega$, and write
\[
\Omega_k \xrightarrow{\gamma} \Omega,
\]
if, for every $f \in L^2(B)$, the weak solution $u_k^f$ of
\[
\begin{cases}
-\Delta u = f, & \text{in } \Omega_k, \\
u = 0, & \text{on } \partial \Omega_k,
\end{cases}
\]
converges strongly in $L^2(B)$ to $u^f$, the weak solution of
\[
\begin{cases}
-\Delta u = f, & \text{in } \Omega, \\
u = 0, & \text{on } \partial \Omega.
\end{cases}
\]
Functions in $H_0^1(\Omega_k)$ are extended by zero outside $\Omega_k$ to elements in $H_0^1(B)$.
\label{defn1}
\end{definition}

\begin{theorem}[{\cite[Theorem~2.5]{henrot}}]
\label{prop1} 
If $\Omega_k \xrightarrow{\gamma} \Omega$, then, for all $j \in \mathbb{N}$,
$\lambda_j(\Omega_k) \to \lambda_j(\Omega)$.
\end{theorem}
Even if Definition \ref{defn1} together with Theorem \ref{prop1} provide a sufficient condition for convergence of eigenvalues, it may be difficult to apply, especially when dealing with sequences of domains defined by geometric transformations. For this reason, we will introduce simple geometric conditions that are more easily verifiable in practice and that ensure $\gamma$-convergence and, consequently, the continuity of the Dirichlet eigenvalues. 

\subsection{Geometric continuity of eigenvalues}
In more practical settings, the study of the continuous dependence of eigenvalues is often facilitated by using Hausdorff convergence. The strategy is the following: $\gamma$-convergence is the right notion for the stability of elliptic
boundary value problems, but it is defined through the solutions of those problems and
is therefore not a geometric condition. Hausdorff convergence, by contrast, is a purely geometric condition on the sets and is easy to check, but on its own it is
too weak. Accordingly, one seeks classes of domains, characterized by a geometric restriction, for which Hausdorff convergence implies \(\gamma\)-convergence.  In general some restrictions are unavoidable (see Example \ref{ball-pit-ex}).
\begin{definition}
    Let $(M,d)$ be a metric space. For each pair of compact, non-empty subsets $X,Y \subset M$, the \emph{Hausdorff distance} between $X$ and $Y$ is defined as
    \[\mathring{d}_H(X,Y) := \max \left\{ \sup_{x \in X} d(x,Y), \ \sup_{y \in Y} d(X,y) \right\},\] 
    where 
    \[d(x,Y) := \inf_{y \in Y} d(x,y), \quad d(X,y) := \inf_{x \in X} d(x,y).\]
\end{definition}
We now consider open subsets of a compact set and define their Hausdorff distance via their complements. In our setting, we consider $M=\R^n$ equipped with the euclidean distance $d(x,y)=|x-y|.$ 
\begin{definition}
Let $\Omega_1,\Omega_2\subset B$ be open sets. Their Hausdorff distance is defined by
$$
d_H(\Omega_1,\Omega_2) := \mathring{d}_H(\B \setminus \Omega_1, \B \setminus \Omega_2),
$$
where the right-hand side is the standard Hausdorff distance between the compact sets $\B \setminus \Omega_1$ and $\B \setminus \Omega_2$.  
\end{definition}
A classical result \cite{henrot} is the following.
\begin{theorem}[Convex case]
Let $\{\Omega_k\}_{k\in\mathbb{N}}$ be a sequence of convex open subsets of $\B$. If $\Omega_k$ converges to an open, convex set $\Omega$ with respect to the Hausdorff distance, then $\Omega_k$ $\gamma$-converges to $\Omega$.
\label{convex}
\end{theorem}
The convexity assumption is, however, rather restrictive. A significant relaxation is obtained by assuming uniform Lipschitz regularity of the domains. Assume that the sets $\Omega_k$ are uniformly Lipschitz, 
meaning that the boundary of $\Omega_n$ is locally the graph of a Lipschitz function, and that we can take a uniform Lipschitz constant for all these Lipschitz functions and for all the sets $\Omega_n$. We refer to \cite{chenais, henrot} for the following result.   
\begin{theorem}[Chenais]
Let $\{\Omega_k\}_{k\in\mathbb{N}}$ be a sequence of uniformly Lipschitz open subsets of $\B$. 
If $\Omega_k$ converges to an open, Lipschitz set $\Omega$ with respect to the Hausdorff distance, then $\Omega_k$ $\gamma$-converges to $\Omega$.
\label{chenais}
\end{theorem}
In dimension two, the assumptions can be further relaxed thanks to a result by Šverák \cite{sverak1}; see also \cite{buttazzo}. In this setting, topological constraints on the domain are sufficient.
\begin{theorem}[Šverák]
Fix $n=2$. Let $\{\Omega_k\}_{k\in\mathbb{N}}$ be a sequence of open subsets of $\B$. Let $p \in \mathbb{N}$ be fixed and assume that the complements $\Omega_k^c$ in $B$ have at most $p$ connected components.  
If $\Omega_k$ converges to an open set $\Omega$ with respect to the Hausdorff distance, then $\Omega_k$ $\gamma$-converges to $\Omega$.
\label{sverak}
\end{theorem}

Finally, the following proposition deals with the
deformations used in the next section. It shows that if a sequence of domains
$\Omega_k$ is obtained from $\Omega$ through maps $\mathrm I+G_k$ converging to
the identity in the $C^1$ norm, then $\Omega_k$ converges to $\Omega$ in the
Hausdorff topology. The rate of convergence is explicit, given by
$\|G_k\|_\infty$.

\begin{theorem}\label{prop29}
Let $B\subset\mathbb{R}^n$ be a bounded open set and let $\Omega$,
$\{\Omega_k\}_{k\in\mathbb{N}}$ be open subsets compactly contained in $B$. Assume that for every
$k$ there exists $G_k\in C^1(\overline B;\mathbb{R}^n)$ such that
\begin{equation}\label{eq:Fk}
  F_k:=(\mathrm I+G_k)\big|_{\overline\Omega}\colon
  \overline\Omega\longrightarrow\overline{\Omega_k}
  \quad\text{is a diffeomorphism mapping }\Omega\text{ onto }\Omega_k .
\end{equation}
Here $\mathrm I$ denotes the identity map and $\overline\Omega$ the closure of
$\Omega$. For $x\in\overline B$ denote by $G_k'(x)$ the Jacobian matrix of
$G_k$ at $x$ and set
\begin{equation}\label{Gknorms}
  \|G_k'(x)\|:=\sup_{|v|=1}|G_k'(x)v| ,\qquad
  \|G_k'\|_{\infty}:=\max_{x\in\overline B}\|G_k'(x)\| ,\qquad
  \|G_k\|_{\infty}:=\max_{x\in\overline B}|G_k(x)| .
\end{equation}
If
\[
\|G_k\|_{C^1} := \|G_k\|_{\infty}+
  \|G_k'\|_{\infty}\longrightarrow0
  \qquad\text{as }k\to\infty,
\]
then $\Omega_k\to\Omega$ in the Hausdorff topology.
\end{theorem}

\begin{proof*}
We set $\varepsilon_k:=\|G_k\|_\infty$. Then
\begin{equation}\label{eq:close}
  |F_k(x)-x|=|G_k(x)|\le\varepsilon_k
  \qquad\text{for every }x\in\overline\Omega .
\end{equation}
Since $F_k$ is a bijection of $\overline\Omega$ onto $\overline{\Omega}_k$
and of $\Omega$ onto $\Omega_k$, and since $\Omega$ and $\Omega_k$ are open,
\begin{equation}\label{eq:bdry}
  F_k[\partial\Omega]
  =F_k[\overline\Omega\setminus\Omega]
  =\overline{\Omega}_k\setminus\Omega_k
  =\partial\Omega_k .
\end{equation}
Let $z\in\overline B\setminus\Omega_k$. We want to show that
$\operatorname{dist}(z,\overline B\setminus\Omega)\le2\varepsilon_k$. If
$z\notin\Omega$, then $z\in\overline B\setminus\Omega$ and there is nothing
to prove. Otherwise $z\in\Omega$, so $F_k(z)\in\Omega_k$ while
$z\notin\Omega_k$. Hence the segment $[F_k(z),z]$ meets $\partial\Omega_k$ at
some point $y$, and by \eqref{eq:close}
\[
  |z-y|\le|z-F_k(z)|\le\varepsilon_k .
\]
By \eqref{eq:bdry}, $y=F_k(\tilde x)$ for some $\tilde x\in\partial\Omega$,
and a second application of \eqref{eq:close} gives
$|y-\tilde x|\le\varepsilon_k$. Since
$\partial\Omega\subset\overline B\setminus\Omega$, we conclude that
\[
  \operatorname{dist}\bigl(z,\overline B\setminus\Omega\bigr)
  \le|z-\tilde x|\le|z-y|+|y-\tilde x|\le2\varepsilon_k .
\]
Now let $z\in\overline B\setminus\Omega$. We show that
$\operatorname{dist}(z,\overline B\setminus\Omega_k)\le2\varepsilon_k$. If
$z\notin\Omega_k$, there is nothing to prove. Otherwise $z\in\Omega_k$, so
$z=F_k(x)$ for some $x\in\Omega$, and $|z-x|\le\varepsilon_k$ by
\eqref{eq:close}. Since $x\in\Omega$ and $z\notin\Omega$, the segment $[x,z]$
meets $\partial\Omega$ at some point $y$ with
$|z-y|\le|z-x|\le\varepsilon_k$. By \eqref{eq:bdry},
$F_k(y)\in\partial\Omega_k\subset\overline B\setminus\Omega_k$, and
$|F_k(y)-y|\le\varepsilon_k$ by \eqref{eq:close}. Therefore
\[
  \operatorname{dist}\bigl(z,\overline B\setminus\Omega_k\bigr)
  \le|z-F_k(y)|\le|z-y|+|y-F_k(y)|\le2\varepsilon_k .
\]
Combining the two estimates we obtain
$d_H(\Omega_k,\Omega)
\le2\varepsilon_k\to0$.
\qed
\end{proof*}
In order to combine Theorem~\ref{prop29} with the continuity results
above, we have to show that uniform Lipschitz regularity is preserved by
small $C^1$ deformations. This is most easily done through the following
equivalent characterization of uniformly Lipschitz domains.
For $y\in\mathbb{R}^n$, a unit vector $\xi$, $\theta\in(0,\pi/2)$ and $h>0$ let
\[
  C(y,\xi,\theta,h):=\bigl\{z\in\mathbb{R}^n:\ 0<|z-y|<h,\
  (z-y)\cdot\xi>|z-y|\cos\theta\bigr\}.
\]
An open set $\Omega$ has the uniform cone property with parameters
$(r,\theta,h)$ if for every $x\in\partial \Omega$ there is a unit vector $\xi_x$
such that $C(y,\xi_x,\theta,h)\subset \Omega$ for every
$y\in\overline \Omega\cap B(x,r)$. A bounded open set has Lipschitz boundary if and
only if it has the uniform cone property for some parameters, and a family of
open sets sharing the same parameters is uniformly Lipschitz
\cite{chenais,HenrotPierre}.
\begin{lemma}\label{lem:unifLip}
Under the assumptions of Theorem~\ref{prop29}, assume that $\Omega$ has
Lipschitz boundary. Then there exists $k_0$ such that the sets
$\{\Omega_k\}_{k\ge k_0}$ are uniformly Lipschitz.
\end{lemma}

\begin{proof*}
Since $\Omega$ has Lipschitz boundary, it has the uniform cone property with
some parameters $(r,\theta,h)$.
Set $\varepsilon_k:=\|G_k\|_\infty$, $\delta_k:=\|G_k'\|_\infty$,
$s:=\tfrac12\sin(\theta/2)$, and choose $k_0$ such that \[
  \varepsilon_k\le\frac r4
  \qquad\text{and}\qquad
  \delta_k\le\frac{s}{1+s}<1
  \qquad\text{for every }k\ge k_0 .
\]
From now on we fix $k\ge k_0$. Let $y\in\overline\Omega$ and $\xi$ be such that
$C:=C(y,\xi,\theta,h)\subset\Omega$. 

In the first step of the proof we claim that
\[C(F_k(y),\xi,\theta/2,h/2)\subset\Omega_k.\] To prove the claim, let $w$ belong to the latter cone. We have to find $x\in\Omega$ with $F_k(x)=w$, that is, a
fixed point of
\[
  \Phi(x):=w-G_k(x).
\]
Set $x_0:=w-G_k(y)$ and $d:=|x_0-y|=|w-F_k(y)|$. Since $w$ belongs to the latter cone, $|w-F_k(y)|\in(0,h/2)$, moreover the vector $x_0-y$
forms an angle of at most $\theta/2$ with $\xi$, so the distance from $x_0$ to
$\mathbb{R}^n\setminus C$ is at least $d\sin(\theta/2)$. It follows from our choice of $s$ that the closed ball
$D:=\overline{B(x_0,sd)}$ is contained in $C$. Since $C$ is convex and
$C\cup\{y\}\subset\overline\Omega$, the mean value inequality gives, for
$x\in D$,
\[
  |(w-G_k(x))-x_0|=|G_k(y)-G_k(x)|\le\delta_k|x-y|
  \le\delta_k(1+s)d\le sd .
\]
Thus $\Phi$ maps $D$ into itself and since its Lipschitz constant is at most $\delta_k<1$, it is a contraction there. By the Banach fixed point theorem, there exists a unique fixed point $x\in D$ and that point satisfies $x\in\Omega$ and $F_k(x)=w$, so $w\in\Omega_k$.

Next, we show the uniform cone condition for $\Omega_k$ with $k\ge k_0$. Let $x'\in\partial\Omega_k$ and $y'\in\overline{\Omega}_k$
with $|y'-x'|<r/2$. By \eqref{eq:Fk} and \eqref{eq:bdry},
$x:=F_k^{-1}(x')\in\partial\Omega$ and $y:=F_k^{-1}(y')\in\overline\Omega$,
and by \eqref{eq:close}
\[
  |y-x|\le|y-y'|+|y'-x'|+|x'-x|<2\varepsilon_k+r/2\le r .
\]
Therefore $C(y,\xi_x,\theta,h)\subset\Omega$, and the first step yields
$C(y',\xi_x,\theta/2,h/2)\subset\Omega_k$. Hence every $\Omega_k$ with
$k\ge k_0$ has the uniform cone property with parameters
$(r/2,\theta/2,h/2)$, independent of $k$.
\qed
\end{proof*}

Finally, we show a similar result for $\Omega$ being a Jordan domain, i.e., a domain whose boundary $\partial\Omega$ is a Jordan curve. Recall that a Jordan curve is the image of a continuous injective map from $\mathbb{S}^1$ into $\mathbb{R}^2$, and by the Jordan curve theorem, a Jordan curve separates $\mathbb{R}^2$ into two connected components, one bounded and one unbounded.

\begin{lemma}\label{lem:jordan}
Let $n=2$ and assume \eqref{eq:Fk}. If $\Omega$ is a Jordan domain, that is,
$\Omega$ is connected and $\partial\Omega$ is a Jordan curve, then every
$\Omega_k$ is a Jordan domain and $\partial\Omega_k=F_k[\partial\Omega]$.
\end{lemma}

\begin{proof*}
Since $F_k$ is a bijection of $\overline\Omega$ onto $\overline{\Omega_k}$
mapping $\Omega$ onto $\Omega_k$, and since $\Omega$ and $\Omega_k$ are
open,
\[
  F_k[\partial\Omega]=F_k\bigl[\overline\Omega\setminus\Omega\bigr]
  =\overline{\Omega_k}\setminus\Omega_k=\partial\Omega_k .
\]
As $F_k$ is a homeomorphism, it maps $\partial\Omega$ homeomorphically onto
$\Gamma_k:=\partial\Omega_k$. Hence $\Gamma_k$ is homeomorphic to
$\partial\Omega$, and therefore to $S^1$: it is a Jordan curve. Moreover
$\Omega_k=F_k[\Omega]$ is connected, being the continuous image of a
connected set. 
It remains to check that $\Omega_k$ is the interior region of $\Gamma_k$.
By the Jordan curve theorem, $\mathbb{R}^2\setminus\Gamma_k$ has exactly two
connected components, a bounded one $I_k$ and an unbounded one $E_k$, and
$\partial I_k=\partial E_k=\Gamma_k$. Being open, $\Omega_k$ is disjoint
from its boundary $\Gamma_k$; being connected, it is contained in one of
the two components, say $C$. Then $\Omega_k$ is open in $C$, and it is also
closed in $C$: if $z\in C\cap\overline{\Omega_k}$ and $z\notin\Omega_k$,
then $z\in\overline{\Omega_k}\setminus\Omega_k=\Gamma_k$, which is disjoint
from $C$. Since $C$ is connected and $\Omega_k\neq\emptyset$, we conclude
that $\Omega_k=C$, and $C=I_k$ because $\Omega_k$ is bounded.
\qed
\end{proof*}
\begin{theorem}
\label{cor:eig}
Under the assumptions of Theorem~\ref{prop29}, suppose moreover that
$\Omega$ satisfies one of the following assumptions:
\begin{enumerate}[label=\roman*)]
    \item $\Omega$ is convex;
    \item $\Omega$ is Lipschitz;
    \item $n=2$ and $\Omega$ is a Jordan domain.
\end{enumerate}
Then $\Omega_k$ $\gamma$-converges to $\Omega$, and
\[
    \lambda_j(\Omega_k)\longrightarrow\lambda_j(\Omega)
    \qquad\text{for every } j\in\mathbb{N}.
\]
\end{theorem}

\begin{proof*}
    The proof of $i)$ is a direct consequence of $ii)$ since every bounded, convex and open domain is Lipschitz. The proof of $ii)$ is a consequence of Lemma \ref{lem:unifLip} together with Theorem \ref{chenais}. Finally, $iii)$ is given by Lemma \ref{lem:jordan} together with Theorem \ref{sverak}.
    \qed
\end{proof*}
\subsection{Convergence of eigenvalues of divergence form operators} \label{ss:convg_divform} 

In this subsection we study the continuity of eigenvalues for the more general elliptic operator which acts on a function $u$ via $-\operatorname{div}(G\nabla u)$ as in \eqref{eq:nonsymG_op}. We start from the Poisson problem
\begin{equation}
  \begin{cases}
 -\operatorname{div}(G\nabla u)=f, & \text{in }\Omega,\\
u=0, & \text{on }\partial\Omega.
\end{cases}
    \label{genlap}
\end{equation}
Above, $\Omega$ is an open subset of $B$.  The first equation holds in the distributional sense, and we assume $f \in L^2(\Omega)$.

\begin{definition}\label{G_sym}
Let $\{\Omega_k\}_{k\in\mathbb{N}}$ be a sequence of open subsets of $B$, and let $\Omega \subset B$ be an open set. Let
\[
L_Gu:=-\operatorname{div}(G\nabla u),
\]
where $G=(g_{ij})$ is a symmetric, uniformly elliptic $n \times n$ matrix, in the sense that there exists $m>0$:
\begin{align}
&g_{ij}\in L^\infty(B) \qquad{\textrm{for all $i$ and $j$}},\\ 
\label{unifell}
 &  
\sum_{i,j=1}^{n}  g_{ij}(x)\xi_i \xi_j \ge m|\xi|^2 \qquad \text{for a.e. } x\in B, \ \xi\in \mathbb{R}^n. 
\end{align}
We say that $\Omega_k$ \emph{$\gamma_G$-converges} to $\Omega$, and write $\Omega_k \xrightarrow{\gamma_G} \Omega,$ if, for every $f\in L^2(B)$, the weak solution $u_k^f\in H_0^1(\Omega_k)$ of
\[
\begin{cases}
 -\operatorname{div}(G\nabla u)=f, & \text{in }\Omega_k,\\
u=0, & \text{on }\partial\Omega_k,
\end{cases}
\]
converges strongly in $L^2(B)$ to $u^f$, the weak solution of
\[
\begin{cases}
 -\operatorname{div}(G\nabla u)=f, & \text{in }\Omega,\\
u=0, & \text{on }\partial\Omega.
\end{cases}
\]
Functions in $H_0^1(\Omega_k)$ are extended by zero outside $\Omega_k$ and regarded as elements of $H_0^1(B)$.
\end{definition}

As for the Laplacian,
the operator $L_G$ has a discrete spectrum on every open set $\Omega\subseteq B$.
Indeed, the bilinear form
\begin{equation}
  \label{eq:forG}
  a_{G,\Omega}:H^1_0(\Omega)\times H^1_0(\Omega)\rightarrow\R,\qquad
  a_{G,\Omega}(u,v)=\int_\Omega G\nabla u\cdot\nabla v
\end{equation}
is symmetric and, by \eqref{unifell}, continuous and coercive on
$H^1_0(\Omega)$. Along the same lines as in the Laplacian case, one can show that its spectrum consists of a sequence of eigenvalues of finite multiplicity.
\[
  0<\lambda^G_1(\Omega)\le\lambda^G_2(\Omega)\le\dots\uparrow\infty ,
\]
repeated according to multiplicity. 
For $G=\mathrm I$ we recover the Dirichlet eigenvalues $\lambda_j(\Omega)$ of
\eqref{eq:minmax}. The Courant--Fischer min--max principle applied to $a_{G,\Omega}$ yields the analogous monotonicity property for the eigenvalues on nested domains.

We aim to show that Šverák's Theorem \ref{sverak} naturally extends to the more general class of elliptic operators in divergence form. For this purpose, we consider the notion of Mosco \cite{mosco} convergence for function spaces, which provides the natural analogue of $\gamma$-convergence at the level of the associated Sobolev spaces.
\begin{definition}
    Let \(X\) be a Banach space and \((Y_k)_{k\in\mathbb{N}}\) a sequence of subsets of \(X\). The weak upper and the strong lower limits in the sense of Kuratowski are defined as follows:
\begin{align*}
w\text{-}\limsup_{k\to\infty} Y_k
&=
\left\{
u\in X :
\exists (k_h)_h,\;
\exists\, u_{k_h}\in Y_{k_h}
\text{ such that }
u_{k_h}\rightharpoonup u
\right\},\\[1ex]
s\text{-}\liminf_{k\to\infty} Y_k
&=
\left\{
u\in X :
\exists\, u_k\in Y_k
\text{ such that }
u_k \to u
\right\}.
\end{align*}

If $Y,\{Y_k\}_{k\in\mathbb{N}}$ are closed subspaces of $X$, we say that $Y_k$ converges to $Y$ in the sense of Mosco and we write $Y_k\xrightarrow{M}Y$ if
\begin{align*}
\text{(M1)}\qquad
&Y
\subseteq
s\text{-}\liminf_{k\to\infty} Y_k,\\
\text{(M2)}\qquad
&w\text{-}\limsup_{k\to\infty} Y_k
\subseteq Y.
\end{align*}
\end{definition}
Condition (M1) expresses an approximation property: every element of the limit space $Y$ can be approximated strongly by a sequence of elements $u_k\in Y_k$. Condition (M2) is a stability property: every weak limit in $X$ of a sequence $u_k\in Y_k$ necessarily belongs to $Y$. Together, these conditions ensure that the sequence of spaces converges to $Y$ without losing any element of the limit space nor introducing additional ones.
\begin{lemma}
Let $X$ be a Hilbert space and let $\|\cdot\|_1$, $\|\cdot\|_2$ be equivalent norms on $X$, i.e., there exist constants $c, C > 0$ such that
\[
c \|x\|_1 \le \|x\|_2 \le C \|x\|_1 \quad \forall x \in X.
\]
If a sequence of closed subspaces $\{Y_k\}$ Mosco-converges to $Y$ with respect to $\|\cdot\|_1$, then it also Mosco-converges to $Y$ with respect to $\|\cdot\|_2$.
\label{lem:equivalence}
\end{lemma}
\begin{proof*}
Since the norms are equivalent, they induce the same strong and weak topologies on $X$. Indeed, strong convergence is preserved by norm equivalence, while weak convergence is unchanged since the dual space (as a set) and the notion of continuity of linear functionals do not depend on the choice of an equivalent norm. Hence, both Mosco conditions (M1) and (M2) are independent of the choice of equivalent norm, yielding the claim.
\qed
\end{proof*}
 
The following proposition says that
$\gamma_G$-convergence does not depend on $G$ at all.
 
\begin{theorem}\label{prop:gammaG}
Let $\{\Omega_k\}_{k\in\N}$ and $\Omega$ be open subsets of $B$, and let $G$ be defined as above. Then
\[
  \Omega_k\gconv\Omega
  \qquad\Longleftrightarrow\qquad
  \Omega_k\gGconv\Omega .
\]
\end{theorem}
 
\begin{proof*}
Let $H^1_0(\Omega)\hookrightarrow H^1_0(B)$ by extension to zero for any open
set $\Omega\subseteq B$. We define on $H^1_0(B)\times H^1_0(B)$
\[
  a(u,v)=\int_B\nabla u\cdot\nabla v\ ,
  \qquad
  a_G(u,v)=\int_B G\nabla u\cdot\nabla v\  .
\]
By the uniform ellipticity of $G$ there exist some positive
constants $c,C$,
\[
  c\,a(u,u)\le a_G(u,u)\le C\,a(u,u),
  \qquad\forall\,u\in H^1_0(B),
\]
so that the norm $\|u\|_G=a_G(u,u)^{1/2}$ is equivalent to the standard norm
\[
  \|u\|_{H^1_0(B)}=a(u,u)^{1/2}=\|\nabla u\|_{L^2(B)} .
\]
Equivalent norms define the same space $H^1_0(B)$ as a topological space.
By \cite[Thm.~4.5.4]{buttazzo}, $\Omega_k\gconv\Omega$ if and only if
\[
  H^1_0(\Omega_k)\Mconv H^1_0(\Omega).
\]
By Lemma~\ref{lem:equivalence}, the Mosco convergence also holds with respect
to $\|\cdot\|_G$. By \cite[Theorem 3.3]{daners03} and \cite[Theorem 2.3.10]{henrot}, we conclude that
\[
  \pushQED{\qed}
  \Omega_k\gconv\Omega
  \qquad\Longleftrightarrow\qquad
  \Omega_k\gGconv\Omega .
  \qedhere
  \popQED
\]
\end{proof*}
 Theorem~\ref{prop:gammaG} says that $\gamma$-convergence is a property of
the sequence of spaces $H^1_0(\Omega_k)$ inside $H^1_0(B)$, and not of the
differential operator: the operator enters only through the equivalence of the
forms $a$ and $a_G$, which the ellipticity of $G$ guarantees. Every criterion established for $\gamma$-convergence transfers to $\gamma_G$ including the convex
case in Theorem \ref{convex}, Chenais's Theorem \ref{chenais} and \v{S}ver\'ak's Theorem \ref{sverak}. We record the last one explicitly.
 
\begin{theorem}[\v{S}ver\'ak for the general problem in divergence form]
\label{th:sverak_G}
Let $n=2$ and let $\{\Omega_k\}_{k\in\N}$ be a sequence of open subsets of
$B\subseteq\R^2$. Let $p\in\N$ be fixed and assume that the complements
$\Omega_k^c$ in $B$ have at most $p$ connected components. If $\Omega_k$
converges to an open set $\Omega$ with respect to the Hausdorff distance, then
\[
  \Omega_k\gGconv\Omega .
\]
\end{theorem}

\begin{proof*}
By \ref{sverak} the hypotheses give $\Omega_k\gconv\Omega$, and
Theorem~\ref{prop:gammaG} turns this into $\Omega_k\gGconv\Omega$.
\qed
\end{proof*}
\begin{theorem}\label{prop:eigG}
Let $\{\Omega_k\}_{k\in\N}$ and $\Omega$ be open subsets of $B$. If $\Omega_k\gGconv\Omega$,
then \[\lambda^G_j(\Omega_k)\to\lambda^G_j(\Omega) \text{ for every j }\in\N.\]
\end{theorem}

\begin{proof*}
By definition, $\gamma_G$-convergence is the convergence of the resolvents of $L_G$, so the
proof of Theorem~\ref{prop1} works along the exact same lines, with the form $a$ replaced by $a_G$; see
also \cite[Thm.~5.2.4, Cor.~5.2.5 and Cor.~4.3.2]{daners}, where the argument is carried out
for general elliptic operators in divergence form.
\qed
\end{proof*}

\subsection{Eigenvalue convergence and non-convergence on rough domains} \label{ss:examples_conv_not}

    Šverák's theorem in dimension $2$ is quite general. In particular, it applies to any sequence of domains whose boundaries are Jordan curves, even with very low regularity. This class of curves includes figures that can hardly be called curves in common sense, as they are impossible to draw; one may only draw approximations thereof. 
    
   \subsubsection{Eigenvalue convergence in case of a fractal boundary} 
    
    For example, we consider the fractal curve known as Koch snowflake.
Starting from a unit side-length equilateral triangle, we define a sequence of polygonal Jordan curves $\{\Gamma_k\}_{k\in\mathbb{N}}$ as follows. Given $\Gamma_k$, we construct $\Gamma_{k+1}$ by replacing each line segment of $\Gamma_k$ with four segments according to the following procedure:
\begin{enumerate}
    \item Divide the segment into three parts of equal length;
    \item Construct an equilateral triangle, pointing outward, on the middle subsegment;
    \item Remove the middle subsegment of the original segment.
\end{enumerate}

\begin{figure}[h!]
	\centering
	\includegraphics[width=1\textwidth]{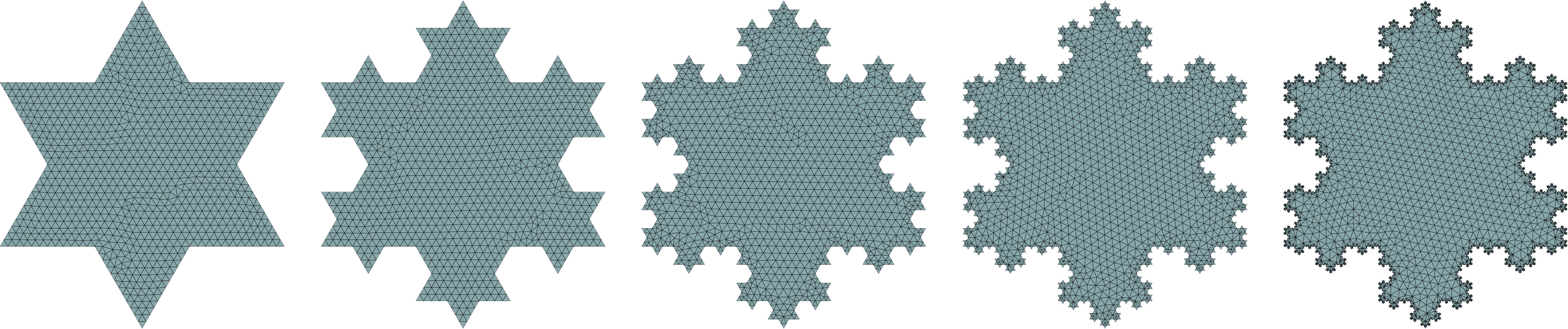}
	\caption{The first five iterations of the Koch snowflake.}
	\label{fractal-disc}
\end{figure}

In this way, each segment is replaced by four segments of equal length as depicted in Figure \ref{fractal-disc}. Observe that $\Gamma_{k+1}$ is again a closed polygonal curve.
We define the Koch snowflake curve as the limit $
\Gamma := \lim_{k\to\infty} \Gamma_k$,
which is a Jordan curve.
Finally, we define the Koch snowflake domain as $\Omega := \bigcup_{k=1}^{\infty} \Omega_k$,
where $\Omega_k := \operatorname{int}\Gamma_k$ denotes the bounded connected open set in $\mathbb{R}^2$ such that $\partial \Omega_k = \Gamma_k$. We claim that $\Omega_k \xrightarrow{H} \Omega$. Consider $B$ such that $\overline{\Omega}\subset B$. 
Observe that $B\setminus\Omega_2\subseteq B\setminus\Omega_1$ thus $\sup_{y \in B\setminus\Omega_2} d(B\setminus\Omega_1,y)=0$. On the other hand, suppose that $x\in (B\setminus\Omega_1)\cap (B\setminus \Omega_2)$, then
\[ d(x,\B\setminus\Omega_2)=0.\]
If $x\in \B\setminus\Omega_1$ but $x\not \in \B\setminus \Omega_2$ then $x$ is in one of the equilateral triangles of side length $1/3$ built on the sides of $\Omega_1$. Thus, by the Pythagorean Theorem
\[\sup_{x \in \B\setminus \Omega_1} d(x,\B\setminus\Omega_2)\le \frac{1}{3}\frac{\sqrt{3}}{2}\quad \Rightarrow \quad d_H(\Omega_1,\Omega_2)\le  \frac{1}{3}\frac{\sqrt{3}}{2}.\]
By induction
\[d_H(\Omega_k,\Omega_{k+1})\le  \frac{1}{3^k}\frac{\sqrt{3}}{2}.\]
Finally, for an integer $m>k$, the triangle inequality gives
\[
  d_H(\Omega_k,\Omega_{m})\le\sum_{r=k}^{m-1}d_H(\Omega_r,\Omega_{r+1})
  \le\sum_{r=k}^{\infty}\frac{1}{3^{r}}\frac{\sqrt3}{2}
  =\frac{1}{3^{k}}\frac{3\sqrt3}{4}.
\]
The sets $\overline B\setminus\Omega_m$ decrease, as $m\to\infty$, to
$\overline B\setminus\Omega$, and a decreasing sequence of nonempty compact
sets converges in the Hausdorff distance to its intersection; hence
$d_H(\Omega_k,\Omega_m)\to d_H(\Omega_k,\Omega)$ as $m\to\infty$, and the
estimate above passes to the limit:
\[
  d_H(\Omega_k,\Omega)\le\frac{1}{3^{k}}\frac{3\sqrt3}{4},
  \qquad\text{so that}\qquad
  \lim_{k\to\infty}d_H(\Omega_k,\Omega)=0 .
\]
Since each $\Gamma_k=\partial\Omega_k$ is a Jordan curve, they each divide $\mathbb{R}^2$ into two connected regions and since $\operatorname{dist}(\Gamma_k,\partial B)>0$ for all $k$, they each divide $B$ into two connected regions. We can therefore apply Šverák's theorem and conclude that $\Omega_k\xrightarrow{\gamma}\Omega$. By Theorem \ref{prop1}, $\lambda_j(\Omega_k)\to\lambda_j(\Omega)$ for every $j$, i.e. the eigenvalues of the snowflake are the limits of those of its polygonal approximations.

\subsubsection{Hausdorff convergence without eigenvalue convergence}
We conclude this subsection with some examples of pathological domains where the continuity of eigenvalues fails.  

In our first example of Hausdorff convergence without eigenvalue convergence, the sequence of planar domains Hausdorff converges to a segment. This example violates the hypothesis of Theorem \ref{th:sverak_G} because the limit set is not an open subset of the plane.
\begin{example}[Collapsing domains]
Consider an open rectangular domain in two dimensions
\[
R = (0,b_1) \times (0,b_2)\subset \B.
\]
For each $k \in \mathbb{N}$, define the collapsing domains $ 
R_k = I \times J_k = (0,b_1) \times (0, b_2/k),$
where $I = (0,b_1)$ and $J_k = (0,b_2/k)$. On $B_k$, the Dirichlet Laplacian has explicitly computable eigenvalues
\[
\lambda_{m,n}(R_k) = \pi^2 \left( \frac{m^2}{b_1^2} + \frac{n^2 k^2}{b_2^2} \right), \quad m,n \in \mathbb{N}.
\]

We identify the segment $I$ with $I\times\{0\}$. Since it is a subset of $R_k$ \[\sup_{x \in \overline{B}\setminus R_k} d(x,\overline{B}\setminus I)=0\quad  \text{and} \quad \sup_{y \in B\setminus I } d(\overline{B}\setminus R_k,y)\le b_2/k\quad \Rightarrow \lim_{k\to\infty}d_H(R_k,I)=0.\] 
On the other hand, the eigenvalues satisfy
\[
\lim_{k \to \infty} \lambda_{m,n}(R_k) = \infty \quad \text{for all } m,n \ge 1,
\]
while the Dirichlet Laplacian on the one dimensional limit set $I$ has finite eigenvalues $
\lambda_m = \frac{\pi^2 m^2}{b_1^2}.$
This case is implicitly excluded by our results in the previous section simply by requiring that the limiting set under Hausdorff convergence is an open set with non-empty interior. 
\end{example}

The next example provides another instance of eigenvalue non-convergence: the domains converge in the Hausdorff distance to a limit with non-empty interior, yet the eigenvalues fail to converge to those of the limit.  In this example the hypothesis of Theorem \ref{th:sverak_G} are not satisfied because the number of connected components of the complements tends to infinity.  To establish this example, we recall the definition of capacity and
the three properties we shall use.
\begin{definition} Let $E\subset D\subseteq\R^n$. Assume that the closure of $E$ is a compact subset of $D$. The capacity of $E$ in $D$ is
    \[
    \operatorname{cap}(E,D)
    =
    \inf\left\{
    \int_D |\nabla u|^2 \ 
    :\ 
    u\in \mathcal{U}_E
    \right\},
\]
where $\mathcal{U}_E$ is the set of all functions $u$ in the Sobolev space $H_0^1(D)$ such that
\[
u \geq 1
\quad\text{almost everywhere in an open neighborhood of } E.
\]
Under these assumptions, the distance from the boundary of $\overline{E}$ to the boundary of $D$ is positive and therefore the set $\mathcal U _E $ is non-empty. 
\end{definition}
We refer to \cite{cap1,cap2,cap3} for a detailed theory on capacities and the proof of the following properties.
\begin{proposition}
The capacity satisfies the following properties.
\begin{enumerate}
    \item {(Monotonicity).}
    Let $F\subseteq E\subseteq D\subseteq\mathbb{R}^n$. Then
    $
    \operatorname{cap}(F,D)\le \operatorname{cap}(F,E).
    $

    \item {(Sub-additivity).}
    Let $B_1,\ldots,B_r$ be domains in $\mathbb{R}^n$, and let
    \[
    E=\bigcup_{i=1}^r B_i\subseteq D, \quad \textrm{ then }
    \operatorname{cap}(E,D)\le
    \sum_{i=1}^{r}\operatorname{cap}(B_i,D).
    \]
    \item (Positivity in dimension two).
    Suppose that $n=2$. Then, for every smooth curve
    $E$ compactly contained in $D$,
    $
    \operatorname{cap}(E,D)>0.
    $
    Moreover, if $\operatorname{cap}(E,D)>0$, then
    $H_0^1(D\setminus E)\neq H_0^1(D).$
\end{enumerate}
\end{proposition}
It is also worth recording one explicit computation we shall need. For $\rho<R$ the capacity is 
     \[\operatorname{cap}(B_\rho,B_R) = \inf \int_{B_R\setminus\bar{B}_\rho} |\nabla u|^2 \]
     where we may assume that  $u|_{B_\rho}=1$ and $u|_{\partial B_R}=0$. The capacity is achieved by some $u$ which, being a critical point for the functional,
    must satisfy $\Delta u(r) = 0$ for $\rho < r < R$. By radial symmetry, the Laplace equation reduces to $(r u')' = 0$, and boundary conditions yield the unique minimizer 
    \[ u(r) = \begin{cases} \frac{\log(R/r)}{\log(R/\rho)}, & \rho \leq r \leq R, \\ 1, & 0\leq r < \rho. \end{cases} \]  
    Integrating  $|\nabla u|^2$ in polar coordinates gives
    \[\operatorname{cap}(B_\rho,B_R)=\int_{0}^{2\pi} \int_{\rho}^{R} \frac{1}{r \log^2(R/\rho)} dr d\theta = \frac{2\pi}{\log(R/\rho)}.\]

The point of the following lemma is to show that capacitary negligibility implies $\gamma$-convergence. 

\begin{lemma}
    Let $D$ be a bounded domain and $E_k\subset D$ a family of compact sets such that $\operatorname{cap}(E_k,D)\to 0$. Then
    $D\setminus E_k\xrightarrow{\gamma}D.$
    \label{capconv}
\end{lemma}
\begin{proof*}
    Equivalently we show Mosco convergence of the function space
    $H_0^1(D\setminus E_k)\xrightarrow{M}H_0^1(D)$.
    Observe that $(M2)$ is trivial: the ambient space $X=H_0^1(D)$, and any $H_0^1(D\setminus E_k)$ with $D\setminus E_k\subset D$ is a subspace of $H_0^1(D)$ by extension to zero. Thus, if $u_{n_k}\in H_0^1(D\setminus E_k)\subset H_0^1(D)$ and $u_{n_k}\rightharpoonup u$ in $H_{0}^1(D)$ then trivially $u\in H_0^1(D)$. We now show (M1).
Since $E_k$ is a compact subset of the open set $D$, $\operatorname{dist}(E_k,\partial D)>0$ and the class $\mathcal U_{E_k}$ is
non-empty. We may choose
$v_k\in\mathcal U_{E_k}$ with
\[
\int_D|\nabla v_k|^2\ \le\ \operatorname{cap}(E_k,D)+\tfrac1k ,
\]
and an open set $\omega_k$ with $E_k\subset\omega_k\subset D$ and $v_k\ge 1$ a.e.\ in $\omega_k$.
Set $w_k:=\min\{|v_k|,1\}\in H^1_0(D).$

Then $0\le w_k\le 1$, $w_k=1$ a.e.\ in $\omega_k$, and since $|\nabla w_k|\le|\nabla v_k|\text{ a.e.}$
we have
\begin{equation}
   \int_D|\nabla w_k|^2\ \le\ \operatorname{cap}(E_k,D)+\tfrac1k\ \longrightarrow\ 0 .
   \label{eq: tozero}
\end{equation}
By the Poincar\'e inequality, $\|w_k\|_{L^2(D)}\to 0$ as well. Next, we build a sequence of functions in $H_0^1(D)$ as follows.
Let $u\in H^1_0(D)\cap L^\infty(D)$ and set
\[
u_k:=u\,(1-w_k)\ \in\ H^1_0(D)\cap L^\infty(D)
\] 
Since $w_k=1$ a.e.\ in $\omega_k$, we have
$u_k=0$ a.e.\ on the open set $\omega_k\supset E_k$, and this suffices to conclude $u_k\in H^1_0(D\setminus E_k).$
Finally,
\[
\|\nabla u_k-\nabla u\|_{L^2(D)}
\ \le\ \|w_k\nabla u\|_{L^2(D)}+\|u\|_{L^\infty(D)}\|\nabla w_k\|_{L^2(D)}.
\]
The second term tends to zero by \eqref{eq: tozero}.  For the first term, note that $\|w_k\|_{L^2(D)}\to0$ alone is not
enough, since $|\nabla u|^2$ is merely integrable thus we cannot apply the Cauchy-Schwarz inequality. We use $0\le w_k\le1$ together with the density of
$C^\infty_c(D;\mathbb R^n)$ in $L^2(D;\mathbb R^n)$. Given $\varepsilon>0$, pick $g\in C^\infty_c(D;\mathbb R^n)$
with $\|\nabla u-g\|_{L^2(D)}<\varepsilon$; then
\[
\|w_k\nabla u\|_{L^2}\ \le\ \|w_k(\nabla u-g)\|_{L^2}+\|w_k\,g\|_{L^2}
\ \le\ \varepsilon+\|g\|_{L^\infty}\|w_k\|_{L^2}.
\]
Therefore
$\|\nabla u_k-\nabla u\|_{L^2(D)}\to0$, and thus \ $u_k\to u$ strongly in $H^1_0(D)$. Consequently, any $u \in H^1 _0 (D) \cap L^\infty(D)$ can be obtained as a limit of elements in $H_0 ^1 (D \setminus E_k)$.  Since $H_0 ^1(D) \cap L^\infty(D)$ is dense in $H^1 _0 (D)$, the same holds for a generic $u \in H_0 ^1 (D)$ by a standard density argument.
\qed
\end{proof*}

Having established these results for the capacity, we now give an example of a sequence of domains that are Hausdorff convergent to a domain with non-empty interior, but not $\gamma$-convergent.  
\begin{example}[The Ball Pit]
\label{ball-pit-ex}

\begin{figure}[b!]
	\centering
	\includegraphics[width=1\textwidth]{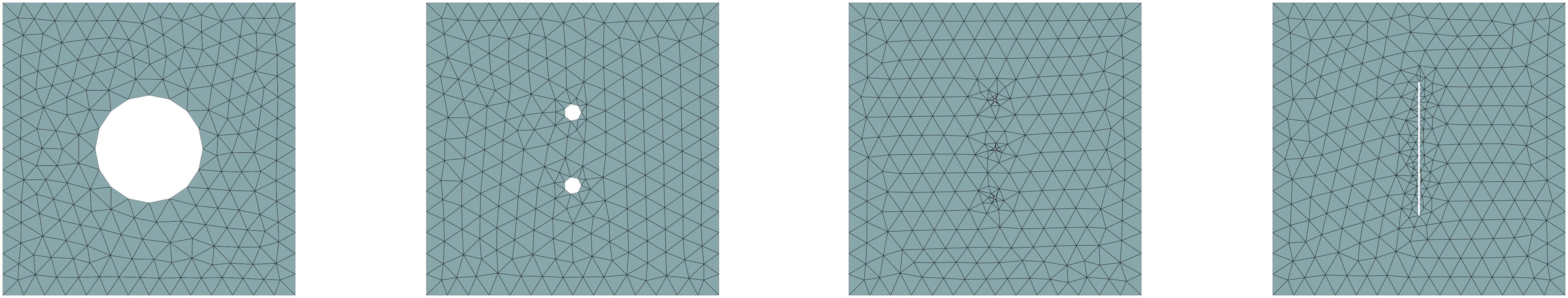}
	\caption{The first three iterations of the Ball Pit followed by their limit set.}
	\label{ball-pit}
\end{figure}

Consider $D=(-1,1)\times(-1,1)$ and $\Sigma=\{0\}\times[-1/2,1/2]$. Let $\Omega=D\setminus\Sigma$. We define for each step $k\geq 1$
\[y_i^k=-\frac{1}{2}+\frac{i-\frac{1}{2}}{k}, \qquad i=1,\dots,k, \qquad x_i^k=(0,y_i^k).\]
The points $x_i^k$ are equidistant points on the segment $\Sigma$ with distance $\varepsilon_k=1/k$ between adjacent points. We define also
\[B_i^k=B(x_i^k,\rho_k), \qquad \rho_k=e^{-k^2}, \qquad \Xi_k=\bigcup_{i=1}^k \overline{B_i^k}, \qquad \Omega_k=D\setminus \Xi_k.\]
Since $2\rho_k<\epsilon_k$ for any $k \geq 1$, $B_i^k\cap B_j^k=\emptyset$ for all $i\neq j$. Observe that
\[\sup_{p\in\Xi_k}d(p,\Sigma)\le 2\rho_k\to0 \qquad\text{and}\qquad \sup_{q\in\Sigma} d(\Xi_k,q)\le \frac{\varepsilon_k}{2}\to0, \textrm{ as } k \to \infty \]
hence $\Omega_k$ Hausdorff converges to $\Omega$.

On the other hand, $\Xi_k$ is the union of $k$ disjoint connected components. Thus the hypotheses of Šverák's theorem are violated. We want to show that in this case there is at least one $j\in \mathbb N$ such that the eigenvalues of the Laplace operator
\[\lambda_j(\Omega_k)\not\rightarrow\lambda_j(\Omega).\] 

    The key idea is that the sequence $\Xi_k$ is capacitarily negligible in the limit, while its Hausdorff limit $\Sigma$ is not:
\[
\operatorname{cap}(\Xi_k,D)\to 0,
\qquad
\operatorname{cap}(\Sigma,D)>0.
\]  
    Since $B(x_i^k,\frac{1}{4})\subset D$, by the monotonicity of the capacity \[\operatorname{cap}\left(B_i^k,D\right)\le \operatorname{cap}\left( B_i^k,B_i\left(x_i^k,\frac{1}{4}\right)\right)= \frac{2\pi}{\log(1/(4\rho_k))} = \frac{2\pi}{k^2 - \log 4}.\]
    
    By sub-additivity of the capacity
    \[\operatorname{cap}(\Xi_k,D)\le \sum_{i=1}^k\operatorname{cap}(B_i^k,D)\le\frac{2\pi k}{k^2-\log4}\xrightarrow{k\to\infty}0.\]
    Hence, by Lemma \ref{capconv}
    \[\Omega_k\xrightarrow{\gamma}D.\]
    But 
    $\operatorname{cap}(\Sigma,D)>0$ by \cite[Theorem 4.17]{evans} because the Hausdorff dimension of $\Sigma$ is one. 
    Thus 
    \[H_0^1(\Omega_k)\xrightarrow{M}H_0^1(D)\neq H_0^1(\Omega)=H_0^1(D\setminus \Sigma),\]
    which means 
    \[
      \Omega_k\not\xrightarrow{\gamma}\Omega.
    \]
So, each $\lambda_j(\Omega_k)$ converges to $\lambda_j(D)$. In principle by some strange miracle it could still be the case that $D$ and $\Omega$ are isospectral.  They are not, which we can see through the short time asymptotic behavior of the heat trace.  Each domain admits a two-term asymptotic expansion.
\[\operatorname{Tr}(e^{-t\Delta_{D}})=\frac{\textrm{Area of } D}{4\pi t}-\frac{\textrm{Perimeter of } D}{8\sqrt{\pi t}}+O(1),\] 
    \[\operatorname{Tr}(e^{-t\Delta_{\Omega}})=\frac{\textrm{Area of } \Omega}{4\pi t}-\frac{\textrm{Perimeter of } \Omega}{8\sqrt{\pi t}}+O(1),\]
    as $t \to 0^+$. The areas of $D$ and $\Omega$ are equal, but their perimeters differ by twice the length of $\Sigma$, which is two.
    Consequently, these heat traces, which are spectral invariants, are not equal, and therefore at least some eigenvalues $\lambda_j(D)$ must differ from $\lambda_j(\Omega)$.   
    We conclude that for at least one $j \in \mathbb N$ the sequence of eigenvalues
    \[\lambda_j(\Omega_k)\not\to\lambda_j(\Omega).\]
\end{example}

\subsection{Continuity of eigenvalues and the rate of convergence for the GMRES method} 
\label{ss:cont_rate_cgn} 

We shall now apply our results for the continuity of eigenvalues of uniformly elliptic divergence form operators, including the Dirichlet Laplacian, under geometric deformations to obtain continuity of the rate of convergence for the GMRES method.

Let $\varepsilon>0$ and consider a one-parameter family of open subsets of a bounded domain $B$,
\[
t \longmapsto \Omega_t, \qquad t \in (-\varepsilon,\varepsilon), \qquad \Omega_0 = \Omega,
\]
which is continuous with respect to the $\gamma_G$-topology. More precisely, we assume that for every $t\in(-\varepsilon,\varepsilon)$ and every sequence $\{t_n\}_{n\in\mathbb{N}}$ converging to $t$, the corresponding sequence of domains satisfies
\[
\Omega_{t_n} \xrightarrow{\gamma_G} \Omega_t .
\]
We note that $\gamma$-convergence corresponds to the case when $G$ is the identity matrix. Our goal is to study the continuity, with respect to the deformation parameter $t$, of the spectral functional
\beq \label{eq:qk_function}
q_k(\Omega_t)
=
\prod_{j=1}^{k}
\frac{1}{\sqrt{ 1 + \, \tilde K \ \lambda_j(\Omega_t)}}.
\eeq 
Here $\lambda_j(\Omega_t)=\lambda^G_j(\Omega_t)$ denotes the $j$-th eigenvalue of the uniformly elliptic divergence form operator $-\ddiv(G\nabla)$ on $\Omega_t$, as defined in \eqref{G_sym}, and $\tilde K=K_0/K_1^2>0$ where we take
\[K_0=2, \quad K_1:=\|\textbf{w}\|_{L^\infty(B)}/\sqrt m.\]
We take $G$ and $\textbf{w}$ defined on all of $B$, thus the constant $\tilde K$ may be chosen independently of $t$, so that the values in \eqref{eq:qk_function} can be compared for different $t$.  As discussed in \S \ref{s:cgn_est1} and \S \ref{s:cgn_est2}, this quantity bounds the rate of convergence of the GMRES method.
As usual, any function $u \in H_0^1(\Omega_t)$ is extended by zero to the whole set $B$, so that
$H_0^1(\Omega_t)$ can be regarded as a closed subspace of $H_0^1(B)$.
Since $\Omega_t \subset B$ for all $t$, we have the continuous embedding
\[
H_0^1(\Omega_t) \subset H_0^1(B),
\]
which immediately yields, by the min--max principle, the monotonicity estimate
\[
\lambda_j(B) \le \lambda_j(\Omega_t),
\qquad \text{for all } j \in \mathbb{N}
\text{ and all } t \in (-\varepsilon,\varepsilon).
\]
In particular, the family $\{\lambda_j(\Omega_t)\}_{t}$ is uniformly bounded from below by the corresponding eigenvalue of the fixed domain $B$.

\begin{theorem} \label{prop:qk}
Let $\{\Omega_t\}_{t\in(-\varepsilon,\varepsilon)}$ be a $\gamma_G$-continuous family of open subsets of $B$.
Then the map
\[
t \longmapsto \bigl(q_k(\Omega_t)\bigr)_{k\in\mathbb{N}}
\]
is continuous from $(-\varepsilon,\varepsilon)$ into $\ell^\infty(\mathbb{N})$, and hence uniformly continuous on every compact subinterval of $(-\varepsilon,\varepsilon)$.
\end{theorem}

\begin{proof*}
Termwise continuity, namely
\[
\forall \,\delta>0\ \exists\,\sigma>0 \ \text{such that } 
|t-s|<\sigma
\ \Longrightarrow\
|q_k(\Omega_t)-q_k(\Omega_{s})|<\delta,
\]
for every fixed $k$, is a direct consequence of the $\gamma_G$ continuity, of Theorem~\ref{prop:eigG} and of the continuity of the function
\[
x \longmapsto 
	\frac{1}{\sqrt{  1 + \, \tilde K \ x }}\,  .
\]
Here, however, our aim is to prove continuity uniform with respect to the index $k$, that is, to fix $t\in(-\varepsilon,\varepsilon)$ and show that
\[
\forall \,\delta>0\ \exists\,\sigma>0 \ \text{such that }
|t-s|<\sigma
\ \Longrightarrow\
\sup_{k\in\mathbb{N}}
\bigl|q_k(\Omega_t)-q_k(\Omega_s)\bigr|
<\delta .
\]
We first control the contribution of large indices. For any $t,s$ and any $k\in\mathbb{N}$, we observe that
\[
\bigl|q_k(\Omega_t)-q_k(\Omega_s)\bigr|
\le q_k(\Omega_t)+q_k(\Omega_s)
\le 2\,q_k(B),
\]
because $\lambda_j(B)\le\lambda_j(\Omega_t)$ for every $j$. Moreover, all the factors in \eqref{eq:qk_function} belong to $(0,1]$, so that $q_k(B)$ is non-increasing in $k$ and
\[
q_k(B)=
\prod_{j=1}^k 
\frac{1}{\sqrt{ 1 + \, \tilde K \ \lambda_j(B)}}  \le \frac{1}{\sqrt{ 1 + \, \tilde K \ \lambda_k(B)}}
\xrightarrow{k\to\infty}0,
\]
since $\lambda_k(B)\to +\infty$ as $k\to\infty$.
Fix now $\delta>0$. There exists $k_0\in\mathbb{N}$ such that
\[
q_{k_0}(B)<\frac{\delta}{4},
\]
and therefore, for every $t$ and $s$,
\[
\sup_{k\ge k_0}
\bigl|q_k(\Omega_t)-q_k(\Omega_s)\bigr|
\le
\sup_{k\ge k_0} 2\,q_k(B)
=
2\,q_{k_0}(B)<
\frac{\delta}{2}.
\]
We point out that this is independent of $t,s$.
We now consider the remaining finite set of indices $k=1,\dots,k_0-1$.
For each of such $k$, the map $t\mapsto q_k(\Omega_t)$ is continuous, and a maximum of finitely many continuous functions is continuous, so there exists $\sigma>0$ such that
\[
|t-s|<\sigma
\ \Longrightarrow\
\max_{k< k_0}
\bigl|q_k(\Omega_t)-q_k(\Omega_s)\bigr|
< \frac{\delta}{2}.
\]
Combining the estimates on the finite part and on the tail, we obtain
\[ |t-s| < \sigma \implies 
\sup_{k\in\mathbb{N}}
\bigl|q_k(\Omega_t)-q_k(\Omega_s)\bigr|
< \delta ,
\]
which is the desired continuity in $t$. The last assertion follows from the Heine-Cantor theorem. In particular $q_k(\Omega_t)$ is continuous at $t=0$, uniformly in $k$.
\qed
\end{proof*}

As a corollary, this shows that the estimated rate of convergence for the GMRES method, controlled by the spectral function $q_k$, is continuous under continuous geometric deformations of domains in the Hausdorff topology. 

\begin{corollary}\label{cor:qk_hausdorff}
    Let $n=2$ and let $\{\Omega_t\}_{t\in(-\varepsilon,\varepsilon)}$ be a family of open subsets of $B \subseteq \R^2$.  Let $p \in \mathbb{N}$ be fixed and assume that the complements $\Omega_t^c$ in $B$ have at most $p$ connected components.  
    If $\Omega_t$ converges to an open set $\Omega_0$ with respect to the Hausdorff distance, as $t \to 0$, then $q_k(\Omega_t) \to q_k(\Omega_0)$ as $t\to 0$, uniformly for all $k$. Here $q_k(\Omega_t)$ is defined in \eqref{eq:qk_function}.
\end{corollary}
\begin{proof*} 
The family of domains satisfies the assumptions of Theorem \ref{th:sverak_G}, so $\Omega_t\gGconv\Omega_0$ as $t\to0$.  The conclusion then follows from the proof of Theorem \ref{prop:qk}, which only uses the $\gamma_G$-convergence at the point under consideration.
\qed
\end{proof*}

\section{Continuous domain dependence of Laplace eigenvalues on the discrete level}
\label{seccdddiscr}

This section is devoted to theoretical results for the continuous dependence of  eigenvalues in the discrete case. We establish two results, one for the case of transformed polygons and one for the approximation of a fixed domain.

\subsection{Continuity of discrete eigenvalues for a smoothly transformed fixed FE subspace}
 \label{sec1}

Let $\Omega \subset \er^2$
be a bounded polygonal domain, and $t_0>0$. Let us consider a family of bounded domains  
$$
\{ \Omega_t\}_{t\in [0, t_0]} \subset \er^2, \qquad \mbox{where} \quad \Omega_0 =: \Omega ,
$$
with the following property:

\sm 

\noindent
{\bf Assumption \ref{sec1}.} 
For any $t\in [0, t_0]$  there exists a diffeomorphism $F_t = I + G_t \, : \overline\Omega\to \overline\Omega_t$, 
where $I$ denotes the identity mapping on $\overline\Omega$, such that, with the notation from \eqref{Gknorms}, 
\begin{equation}
\label{ftcont} 
\|G_t\|_{C^1} := \|G_t\|_{\infty}+
  \|G_t'\|_{\infty}
\to 0
\quad \mbox{as } \, t\to 0.
\end{equation}

\bi

\noindent

In the above setting Theorem \ref{cor:eig}
implies the continuity of the eigenvalues on the operator level. 
Now we derive a discrete analogue of this property.

 Let $V_h\subset \he$ be a given finite dimensional FE subspace with a basis $\vp_1, \dots, \vp_n$. For any  $t\in [0, t_0]$ we define
\begin{equation}
\label{vhtrans}
V_h^t := V_h\circ {F_t}\iv := \{ u_h^t:= u_h \circ {F_t}\iv|\, \ u_h\in V_h  \} \ \subset \het.
\end{equation}
Then  $\vp_1^t, \dots, \vp_n^t$ is a basis in $V_h^t$.
With the notations of 
\eqref{shmhomx}, we let $\bbs $ and $\bbq $ be the stiffness and mass matrices w.r.t. $\Omega_t$, \resp :
$$
{\bigl( \bbs \bigr)}_{i,j}:= \int_{\Omega_t} \nb \vp_i^t\cdot\nb \vp_j^t, \qquad 
{\bigl( \bbq \bigr)}_{i,j}:= \int_{\Omega_t}  \vp_i^t\, \vp_j^t
\qquad (t\in [0, t_0]).
$$
If $\cb\in \er^n$ and $u_h^t:= \sum\limits_{j=1}^n c_j \vp_j^t$, then \
 $$
\bbs  \cb\cdot \cb = \int_{\Omega_t} |\nb u_h^t|^2, \quad 
\bbq  \cb\cdot \cb = \int_{\Omega_t}  {(u_h^t)}^2, 
$$ 
hence the Rayleigh quotients of $\bbs\iv \bbq $ in the $\bbs $-inner product satisfy
\begin{equation}
\label{Rayleigh}
\frac{\bigl\la  \bbs  \iv \bbq \cb, \, \cb\bigr\ra_{\bbs }}{|\cb|^2_{\bbs }} = 
\frac{\bbq  \cb\cdot \cb}{\bbs  \cb\cdot \cb}
=  \frac{\int_{\Omega_t} {(u_h^t)}^2}{\int_{\Omega_t} |\nb u_h^t|^2}\, .
\end{equation}
Let us use the notations
$$
\lambda_j^{h}(\Omega_t):= \lambda_j\bigl(\,  \bbq  \iv \bbs \, \bigr) 
$$
in accordance with \eqref{disceigsqmtx}.

\begin{theorem}
\label{fixed-fe}
	Let Assumption \ref{sec1} hold. Then for any fixed $j=1,...,n$ the $j$th eigenvalues satisfy
	$$
	\lambda_j^{h}(\Omega_t) \rightarrow \lambda_j^{h}(\Omega) \qquad \mbox{as } t\to 0.
	$$
\end{theorem}

\begin{proof*}  Let $j\in\mathbb{N}^+$ be fixed. Using the minimax characterization for the Rayleigh quotients \eqref{Rayleigh},   the discrete inverse  Laplace eigenvalues on $\Omega_t$ satisfy
$$
\begin{array}{l}
\displaystyle
\frac{1}{\lambda_j^{h}(\Omega_t)}=\frac{1}{\lambda_j\bigl(\, \bbq \iv \bbs \, \bigr)}=
\lambda_j\bigl(  \bbs \iv \bbq \, \bigr)
\\[4mm]
\displaystyle = \  \max\limits_{V_j^t\subset V_h^t\atop
	\dim V_j^t=j} \  \min\limits_{u_h^t \in V_j^t \atop u_h^t\neq 0  }   \frac{\int_{\Omega_t} {(u_h^t)}^2}{\int_{\Omega_t} |\nb u_h^t|^2}\, 
=   \max\limits_{V_j\subset V_h\atop
\dim V_j=j} \  \min\limits_{u_h \in V_j \atop u_h\neq 0  }   \frac{\int_{\Omega_t} {(u_h \circ {F_t}\iv)}^2}{\int_{\Omega_t} |\nb (u_h \circ {F_t}\iv)|^2}\,
=: \, \vr_j(t)\qquad (t\in [0, t_0]),
\end{array}
$$
where we used the fact that the  transformation
$u_h^t:= u_h \circ {F_t}\iv$ from \eqref{vhtrans}
gives a one-to-one correspondence between the $j$-dimensional subspaces of $V_h^t$ and $V_h$.  
Moreover, we may transform the integrals from $\Omega_t$ to $\Omega$, using a change of variables with $F_t$:  in general,
\begin{equation}
\label{ell:intsubs}
\int_{\Omega_t} z = \int_{\Omega} (z\circ F_t)\   |\det(F_t')| \, ,
\end{equation}
and using that $\nb (u_h \circ {F_t}\iv) = (\nb u_h \circ {F_t}\iv)\boldsymbol{\cdot}
({F_t}\iv)' =  (\nb u_h \circ {F_t}\iv)\boldsymbol{\cdot}
{({F_t}'\circ {F_t}\iv)}\iv  $, we obtain
$$
\vr_j(t)= \max\limits_{V_j\subset V_h\atop
	dim V_j=j} \  \min\limits_{u_h \in V_j \atop u_h\neq 0  }  
\frac{ \int_{\Omega}  u_h^2\   |\det(F_t')|}{\int_{\Omega} 
	| \nb u_h \boldsymbol{\cdot} (F_t')\iv |^2 \   |\det(F_t')|}\, .
$$ 
For our theorem we must prove that \
$
\vr_j(t)  \to \vr_j(0) \ \ (\mbox{as } t\to 0).
$
We will prove, roughly speaking, that the \corr \ elements of the sets of numbers for $t$ and 0  after the "$\max \min$" are pairwise uniformly close to each other. 
For a given $u_h\in V_h$ and $u_h^t:= u_h \circ {F_t}\iv$, we introduce the notations
$$
N_t(u_h):=   \int_{\Omega}  u_h^2\   |\det(F_t')|, \qquad D_t(u_h):= \int_{\Omega} 
| \nb u_h \boldsymbol{\cdot}( F_t')\iv |^2 \   |\det(F_t')|, \qquad R_t(u_h):=\frac{N_t(u_h)}{D_t(u_h)}
$$
and for the case $t=0$ we just write
$$
N(u_h):= \int_{\Omega} {(u_h)}^2, \qquad D(u_h):= \int_{\Omega} |\nb u_h|^2, \qquad R(u_h):=\frac{N(u_h)}{D(u_h)}\ .
$$
If we prove that there exists a function $\ve: [0, t_0]\to \er^+$ such that for any $u_h\in V_h$
$$
|R_t(u_h)-R(u_h)|\le \ve(t)\to 0 ,
$$ 
then $|\vr_j(t)-\vr_j(0)|\le \ve(t)\to 0$ and we are done.

\sm 

Since $\det(F_0')= \det  I =1$, we may assume that $t_0$ is small enough such that $\det(F_t') >0$, hence we replace $|\det(F_t')|$ by $\det(F_t')$ in the above integrals.
Now, first, 
$$
N_t(u_h)-N(u_h) =   \int_{\Omega}  u_h^2\  \bigl( \det(F_t')-1\bigr) 
=   \int_{\Omega}  u_h^2\  \bigl( \det(I+G_t')-1\bigr) \, .
$$
Here a result from \cite[VII.3 (3)]{GGK}, together with the estimate $\|A\|_1\le d\, \|A\|$  and assumption \eqref{ftcont}
show for any $x\in \Omega$ that
\begin{equation}
\label{deltaggk}
|\det(I+G_t'(x))-1| \le \exp(\|G_t'(x)\|_1)-1 \le \exp(d\|G_t'(x)\|)-1 \le
\exp(d\|G_t'\|_\infty)-1 =:\delta(t)\to 0\, ,
\end{equation}
where  $\|\cdot\|$ is the Euclidean operator norm  
from \eqref{Gknorms} and $\|\cdot\|_1$ is induced by the Manhattan norm. 

Hence
\begin{equation}
\label{NNbecsF}
|N_t(u_h)-N(u_h) | \le   \delta(t)\, \int_{\Omega}  u_h^2 \, = \delta(t)\, N(u_h) , \qquad \mbox{where }\delta(t)\to 0\, .
\end{equation}

Second, using again $|\det(F_t')|=\det(F_t')$, 
$$
D_t(u_h)-D(u_h)   =  \int_{\Omega} \Bigl(  | \nb u_h \boldsymbol{\cdot}( F_t')\iv |^2 \   \det(F_t')  - | \nb u_h   |^2 \Bigr) 
$$
\begin{equation}
\label{terms2term}
=
\int_{\Omega} \Bigl(  | \nb u_h \boldsymbol{\cdot}( F_t')\iv |^2 \    - | \nb u_h   |^2 \Bigr) \, \det(F_t') \ + \ 
\int_{\Omega} | \nb u_h   |^2 \Bigl(    \det(F_t')  - 1 \Bigr) \, .
\end{equation}

For the first term in \eqref{terms2term}, we first observe that 
\begin{equation}
\label{Ktx}
( F_t')\iv = I+K_t, \quad \mbox{where} \quad   \|K_t \|_\infty = \| ( I+G_t')\iv - I\|_\infty   \le \frac{\|G_t'\|_\infty }{1-\|G_t'\|_\infty}
\end{equation}
and the latter inequality follows from the Neumann series expansion,
assuming without restriction that $t_0$ is small enough to imply  $\|G_t'\|_\infty<1$ for all 
$t\in [0, t_0]$. Then we use the estimate
$$
\Bigl|\,  |y\cdot (I+B)|^2 - |y|^2   \Bigr| = 
\Bigl|\,  2\la By, y\ra +|y\cdot B|^2  \Bigr|
\le \|B\| (2+\|B\|) |y|^2
$$
which holds for a general vector $y$ and matrix $B$, and apply it to $y= \nb u_h(x)$ and $B=K_t(x)$ for any $x\in \Omega$. Thus, with \eqref{Ktx}
and with the notation $\kappa(t):= \frac{\|G_t'\|_\infty }{1-\|G_t'\|_\infty} \bigl( 2+\frac{\|G_t'\|_\infty }{1-\|G_t'\|_\infty}  \bigr)$, we obtain 
$$
\Bigl|\,
| \nb u_h \boldsymbol{\cdot}( F_t')\iv |^2 \    - | \nb u_h   |^2\Bigr| \le \kappa(t)| \nb u_h   |^2  \quad \mbox{on}\ \Omega, \qquad \mbox{where} \quad \kappa(t)\to 0.
$$
Furthermore, \eqref{deltaggk} implies for all $x$ that \ \ $\det(F_t'(x))  = |\det(I+G_t'(x))| \le 
1+\delta(t) 
$. \
Hence, denoting $\overline{\delta}(t):=  \kappa(t) (1+\delta(t))$ which tends to 0 as $t \to 0$,
$$
\int_{\Omega} \Bigl|  | \nb u_h \boldsymbol{\cdot}( F_t')\iv |^2 \    - | \nb u_h   |^2 \Bigr| \, \det(F_t') \ \le \kappa(t)\, (1+\delta(t)) \int_{\Omega} | \nb u_h   |^2 \, =\,  \overline{\delta}(t)\, \int_{\Omega} | \nb u_h   |^2 \, .
$$
For the second term in \eqref{terms2term}, we can use \eqref{deltaggk} again, so 
altogether we obtain 
\begin{equation}
\label{DDbecsF}
|D_t(u_h)-D(u_h) | \le 
\bigl( \overline{\delta}(t) +\, \delta(t)   \bigr)
\int_{\Omega} | \nb u_h   |^2 \, =  
\widehat{\delta}(t)\, \int_{\Omega} | \nb u_h   |^2 
= \widehat{\delta}(t)\,
\, D(u_h)\, ,
\end{equation}
where $\widehat{\delta}(t):= \overline{\delta}(t) +\, \delta(t)\to 0$. 
Finally, for any $t\in [0, t_0]$ we have a lower bound
\begin{equation}
\label{Dtbecslow}
D_t(u_h) \ge \mu_0^2 \, \mu_1 \int_{\Omega} | \nb u_h   |^2 \, = \mu_0^2 \, \mu_1 \, D(u_h),
\end{equation}
where $t_0$ can be chosen small enough such that
$$
\mu_0:= \min\limits_{t\in [0, t_0]} \frac{1}{\|F_t'\|_\infty} >0, \qquad 
\mu_1:= \min\limits_{t\in [0, t_0]\atop x\in\overline{\Omega}} |\det(F_t'(x))| >0 \, .
$$

Altogether,
$$
\Big| R_t(u_h)-R(u_h)\Big| \le \frac{\bigl| N_t(u_h)-N(u_h)\bigr| \,  D(u_h)\,  + \ N(u_h)\, \bigl| D(u_h)-D_t(u_h)\bigr| }{D_t(u_h)\,  D(u_h)}
$$

$$
= \frac{ N(u_h)  }{D(u_h)} \ \frac{ D(u_h)  }{D_t(u_h)} \
\Bigl( 
\frac{\bigl| N_t(u_h)-N(u_h)\bigr| }{N(u_h)}\,  + \ 
\frac{  \bigl| D(u_h)-D_t(u_h) \bigr| }{D(u_h) }
\Bigr)\, .
$$
Now, using the Poincaré--Friedrichs inequality $N(u_h)\le C_\Omega D(u_h)$, 
then \eqref{Dtbecslow}, \eqref{NNbecsF}  and  \eqref{DDbecsF}, \resp , 
we obtain
\begin{equation}
\label{finbecs}
\Big| R_t(u_h)-R(u_h)\Big| \le \frac{C_\Omega}{\mu_0^2 \, \mu_1}
\bigl(  \delta(t) + \widehat{\delta}(t)      \bigr) \ =: \ve(t)\, \to 0 \quad \mbox{as } t\to 0,
\end{equation}
where $\ve(t)$ solely depends on Assumption \ref{sec1} and is thus independent of $u_h$. Hence $\vr_j(t)  \to \vr_j(0)$ as $t\to 0$, which was to be proved. 
\qed
\end{proof*}

\begin{remark}
The proof shows uniform convergence in the sense that 
\[
\sup_{\substack{V_h\subset H_0^1(\Omega)\\
1\leq j\leq \dim V_h}}
\left|
\frac{1}{\lambda_j^h(\Omega_t)}
-
\frac{1}{\lambda_j^h(\Omega)}
\right|
\leq \varepsilon(t)
\to 0 
\]
where $\varepsilon(t)$
is  independent of the eigenvalues $j$ and subspaces $V_h$.  
\end{remark}

\subsection{Continuity of eigenvalues for a sequence of refined FE subspaces}
\label{sec-refined}

 In practice, both the domain and its discretization vary: a sequence of
domains $\Omega_n$ approximating $\Omega$ is meshed with increasingly fine
meshes. Properly combining Theorem~\ref{chenais} with Proposition~\ref{propbof1}, we
show that the discrete eigenvalues $\lambda_j^{h_n}(\Omega_n)$ converge to the
continuous eigenvalues $\lambda_j(\Omega)$ of the limit domain.

\begin{theorem}
\label{refined-fe}
Let $\Omega \subset \mathbb{R}^2$ be a bounded domain, and let $\{\Omega_n\}_{n\in\mathbb{N}}\subset \mathbb{R}^2$ be a sequence of convex
polygonal domains satisfying the conditions of Theorem \ref{chenais}, i.e. $\Omega_n$ are uniformly Lipschitz open subsets of a given box $B$, and   $\Omega_n \rightarrow \Omega$ with respect to the Hausdorff distance. 

Consider a sequence $h_n>0$ ($n\in \mathbb{N}^+$), such that $h_n\to 0$, and a shape regular family of finite element discretizations on $\Omega_n$ with first order Courant elements and mesh parameter $h_n$.
 Then, for all fixed indices $j \in \mathbb{N}^+$, with the notations \eqref{conteigsq} and \eqref{disceigsqmtx},
\begin{equation}
\label{th32conv}
	\lambda_j^{h_n}(\Omega_n) \rightarrow \lambda_j(\Omega)
	\qquad \mbox{as } n\to \infty.
 \end{equation}
\end{theorem}

\begin{proof*}
Let us fix an index $j$. We will show the convergence using the Squeeze theorem. 	
The lower bound of the sequence $\lambda_j^{h_n}(\Omega_n)$ comes from
 (\ref{third-ineq28}), which states that for any $n\in\mathbb{N}^+$
	\begin{equation} \label{sqlow}
	\lambda_j^{h_n}(\Omega_n) \geq 	\lambda_{j}(\Omega_n) 
		\text{.}
	\end{equation}
For the upper bound, we use \cite[Section 8]{boffi}, where a fixed domain was considered. Applying its estimation to $\Omega_n$, we obtain
\begin{equation}
\label{squpp}
\lambda_j^{h_n}(\Omega_n)
\leq
\frac{\lambda_{j}(\Omega_n)}{\left(1-\overline{C}(\Omega_n) \, h_n^2\, \lambda_{j}(\Omega_n)\right)^2 }
\end{equation}	
where the constant $\overline{C}(\Omega_n)>0$ comes from the 	
 following interpolation bound:
	\begin{equation}
		\|v-\Pi_h v\|_{L^2(\Omega_n)} \leq \overline{C}(\Omega_n)\,  h_n^2 \|\Delta v\|_{L^2(\Omega_n)}
	\end{equation}
in which $\Pi_h v$ denotes the 	interpolant of $v$ in the FE subspace. We must show that $\overline{C}(\Omega_n)$ is bounded independently of $n$. First, recall the 
interpolation estimate
\begin{equation}
\label{interp}
\|v-\Pi_h v\|_{L^2(\Omega_n)} \leq C_{int}\,  h_n^2  | v |_{H^2(\Omega_n)}
\end{equation}	
where the constant $C_{int}>0$ is independent of $\Omega_n$ for the considered FEM, see, e.g., \cite[section 3.2.2.]{Suli}.
Further, from the convexity of $\Omega_n$,
\begin{equation}
\label{MirTal}
  | v |_{H^2(\Omega_n)} \le \|\Delta v\|_{L^2(\Omega_n)}
\end{equation}	
by the Miranda–Talenti estimate, see, e.g., \cite[Thm 2]{SmSuli}.
Altogether, from \eqref{interp}--\eqref{MirTal}, 
\ $\, \overline{C}(\Omega_n) \le C_{int}$ \ $(\forall n\in \mathbb{N}^+)$.
Hence, from \eqref{sqlow} and \eqref{squpp}, 
\begin{equation}
\label{squeeze}
\lambda_{j}(\Omega_n)
\leq
\lambda_j^{h_n}(\Omega_n)
\leq
\frac{\lambda_{j}(\Omega_n)}{\left(1-C_{int} \, h_n^2\, \lambda_{j}(\Omega_n)\right)^2 }
\text{.}
\end{equation}
Here Theorem \ref{chenais} yields that $\Omega_n$ $\gamma$-converges to $\Omega$, and hence  Theorem 	\ref{prop1} implies that 	
$$
\lambda_j(\Omega_n) \to \lambda_j(\Omega) \qquad \mbox{as \ $n \rightarrow \infty$. }
$$
Using also the assumption $h_n \rightarrow 0$,  we find that both sides of \eqref{squeeze} tend to  $\lambda_j(\Omega)$, hence \eqref{th32conv} holds. 
\qed
\end{proof*}

\begin{remark}
	We expect that  convexity of $\Omega_n$ might be relaxed 
to a uniform upper  bound	$\omega < 2\pi$ for the  angles.
If one could prove  interpolation and regularity estimates (in analogy with \cite{Grisvard, Ern})
$$
\|v-\Pi_h v\|_{L^2(\Omega_n)} \leq \tilde C_{int}\,  h_n^{1+\alpha}  | v |_{H^{1+\alpha}(\Omega_n)}, \qquad | v |_{H^{1+\alpha}(\Omega_n)} \le C_{1+\alpha} \|\Delta v\|_{L^2(\Omega_n)}\qquad (\alpha < \pi/\omega),
$$
\resp , then the above proof would work with replacing $C_{int} \, h_n^2 $ in \eqref{squeeze} by $\tilde C_{int}\, C_{1+\alpha}  \,  h_n^{1+\alpha}$.  Presently we do not know about such estimates with  bounds being uniform in $\Omega_n$. 
\end{remark}
 
\section{Numerical experiments}
\label{secnumexp}

In this section, our aim is to test numerically the domain dependence of the performance of the GMRES/GCR method via the relative residual norms for a wide range of selected domain sequences. 
We will find that the theoretical results are confirmed: we observe in several examples the continuous dependence of the sequence of residual norms when theory ensures continuous dependence of the eigenvalues on the domain, and in turn, we see failure of continuous dependence in the last example when it is expected by the theory.

First, we consider examples of smoothly transformed fixed FE subspaces illustrating the results of section \ref{sec1}, then we continue the experiments with examples of a sequence of refined FE subspaces, which will demonstrate the results of section \ref{sec-refined}. Finally, as further interesting examples, we test the continuous dependence on iterations of the Koch snowflake from section \ref{ss:examples_conv_not}, and show as a last example that continuity is lost for the Ball Pit domain \ref{ball-pit-ex} also in the discretized form.

Throughout the upcoming numerical experiments, we always solve the convection-diffusion problem
\begin{equation}
	\label{model-problem}
	\begin{cases}
		-\Delta u + (\beta,0)\cdot\nabla u = 1 \\
		u|_{\partial\Omega} = 0
	\end{cases}
\end{equation}
where the positive parameter $\beta$ will be specified for each test in such a way that the iterative method takes approximately 20 steps to reach a tolerance of $10^{-8}$ for the relative residual norms.

The FE discretization of the domains was carried out so that the meshes contain at least $15{,}000$ nodes and about twice as many $1^\text{st}$-order triangular Lagrange elements. The meshes were generated by gmsh in Python with a setting to make them as uniform as possible, although they become finer near boundaries consisting of closely located boundary nodes. The discretization process was fully handled by FEniCSx, but the resulting systems of linear equations were solved by a custom-programmed iterative solver following the algorithm of the preconditioned GCR method in section \ref{s:cgn_est1}.

\subsection{Examples of smoothly transformed fixed FE subspaces}
In this section, we fix an initial domain with an FE subspace, and perform a smooth transformation on it. Our first example is the simultaneous rotation and stretching of a regular hexagon, which is then followed by the continuous approximation of a circle and an annulus from both sides. In all three cases, the relative residual norms show continuous and monotonic dependence on the domain. These results support the claim of Theorem \ref{fixed-fe}.

\subsubsection{Rotation and stretch of a regular hexagonal domain}

Let $\Omega$ be the inscribed regular hexagon of the unit disk which has a vertex at $(1,0)$. After creating the mesh, we perform a smooth transformation involving counter-clockwise rotation and stretching. For a visual representation see Figure \ref{cont-hexagon}. The results of the numerical tests with $\beta = 1$ can be seen in Figure \ref{hexagon-cont-graphs}. In the figure, the blue curve represents the norms corresponding to the initial domain $\Omega$, and the red curve corresponds to the final domain. The different steps of the transformation are denoted by the color gradient from blue to red in a continuous way.

\begin{figure}[h!]
	\centering
	\includegraphics[width=1\textwidth]{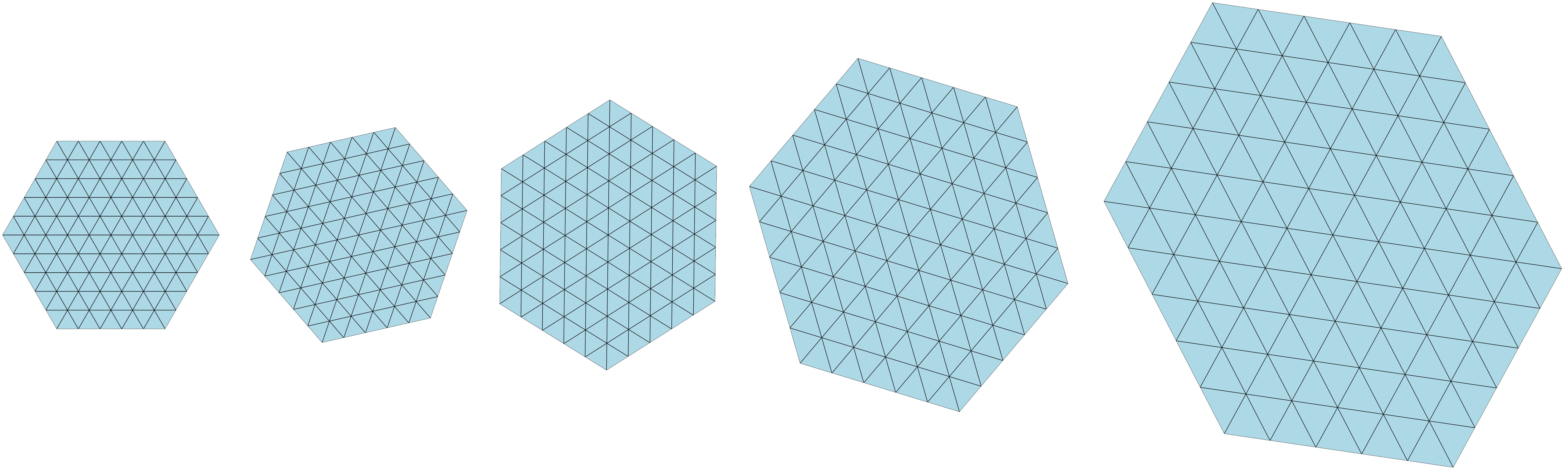}
	\caption{Continuous transformation of the regular hexagon involving rotation and stretching.}
	\label{cont-hexagon}
\end{figure}
\begin{figure}[h!]
	\centering
	\includegraphics[width=0.5\textwidth]{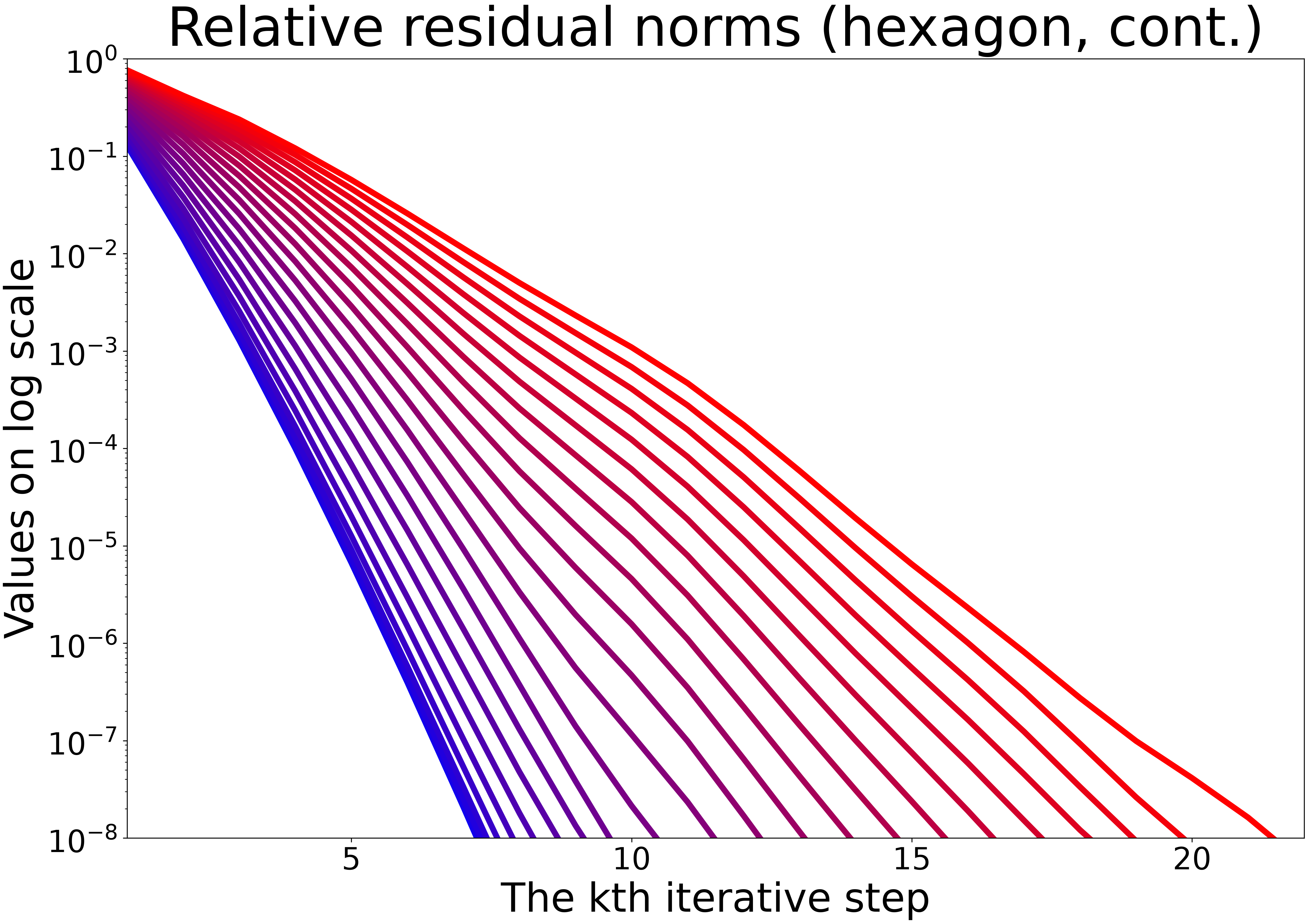}
	\caption{Relative residual norms for the continuous transformation of the hexagon.}
	\label{hexagon-cont-graphs}
\end{figure}

\subsubsection{Continuous approximation of the unit disk from both sides}
Let $\Omega$ be the unit disk. Initially, the domain for the inner approximation is the inscribed square of $\Omega$ which has diagonals coinciding with the axes, and for the outer approximation, it is the circumscribed square with the same positioning. The considered continuous transformation illustrated in Figure \ref{circle-aprx} expands the inscribed square while fixing the 4 vertices and shrinks the circumscribed square while fixing the midpoints of the sides. Meanwhile, the structure of the FE subspace is preserved.

The results of the numerical tests with $\beta=8$ can be seen in Figure \ref{circle-annulus-cont-graphs}. We can observe continuous monotonic behavior as the color gradient from blue to red moves inwards uniformly for the inner (solid) and outer (dashed) approximations as well, which are strictly separated.

\begin{figure}[b!]
	\centering
	\includegraphics[width=1\textwidth]{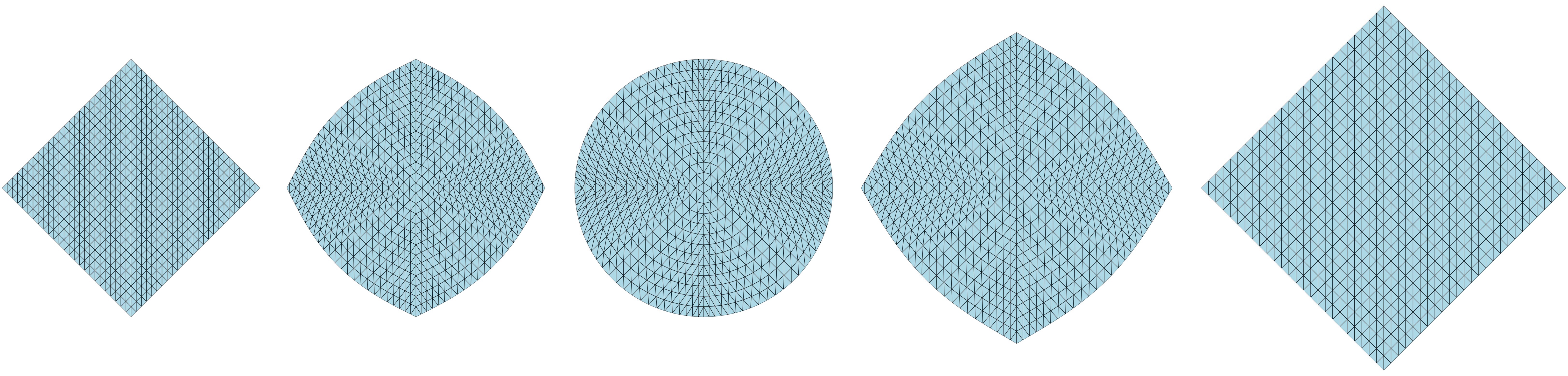}
	\caption{Continuous monotonic approximation of the unit disk from the inside (left) and the outside (right). Animations are available \href{https://ladosbalint.web.elte.hu/cont-dep/square-in.gif}{\textcolor{blue}{here}} and \href{https://ladosbalint.web.elte.hu/cont-dep/square-out.gif}{\textcolor{blue}{here}}.}
	\label{circle-aprx}
\end{figure}
\begin{figure}[b!]
	\centering
	\includegraphics[width=1\textwidth]{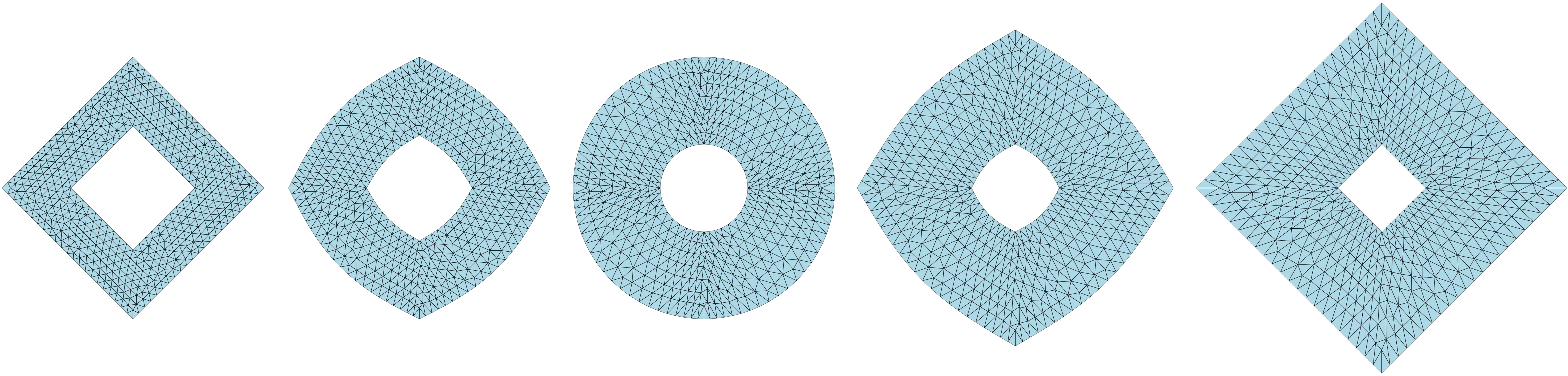}
	\caption{Continuous monotonic approximation of the annulus of radius 1 and 3 from inside (left) and outside (right). Animations are available \href{https://ladosbalint.web.elte.hu/cont-dep/annulus-in.gif}{\textcolor{blue}{here}} and \href{https://ladosbalint.web.elte.hu/cont-dep/annulus-out.gif}{\textcolor{blue}{here}}.}
	\label{annulus-aprx}
\end{figure}

\subsubsection{Continuous approximation of an annulus from both sides}
Let $\Omega$ be the annulus determined by the unit circle and the circle of radius 3 centered at the origin. Initially, the domain for the inner approximation is determined by the circumscribed square of the unit circle and the inscribed square of the circle with radius 3, and vice versa for the outer approximation. The considered continuous transformation illustrated in Figure \ref{annulus-aprx} fixes the vertices of the inscribed squares and the midpoints of the circumscribed squares and expands or shrinks the domain while deforming it into the annulus. Meanwhile, the structure of the
FE subspace is preserved. The results of the numerical tests with $\beta=3$ can be seen in Figure \ref{circle-annulus-cont-graphs}. Overall, the same behavior can be observed as for the previous example.

\subsection{Examples of a sequence of refined FE subspaces}
In this section, we fix a domain $\Omega$ which is approximated by a sequence of domains $\Omega_n$. The generated mesh gets finer as $n$ increases. Our first two examples will be the discrete analogues of the approximation of the circle and the annulus from the previous section. These will be followed by the discrete approximation of an L-shaped and an irregular kidney-shaped domain. Finally, the Koch snowflake will be considered.  In all cases, the relative residual norms show continuous and monotonic dependence on the domain. These results support the claim of Theorem \ref{refined-fe}.

\subsubsection{Discrete approximation of the unit disk from both sides}
Let $\Omega$ be the unit disk. For the inner approximation, let $\Omega_n \;(n \geq 3)$ be the inscribed regular polygon of the unit disk with $n$ vertices that has a vertex placed at $(1,0)$, and for the outer approximation, consider the enlarged versions of these centered at the origin such that they are circumscribed regular polygons of the unit disk. For any fixed $n \geq 3$, the FE subspaces of the inscribed and circumscribed polygons have the same structure and number of nodes and elements. For a visual representation see Figure \ref{circle-disc}. The results of the numerical tests with $\beta = 8$ can be seen in Figure \ref{circle-annulus-disc-graphs}.

\begin{figure}[h!]
	\centering
	\includegraphics[width=1\textwidth]{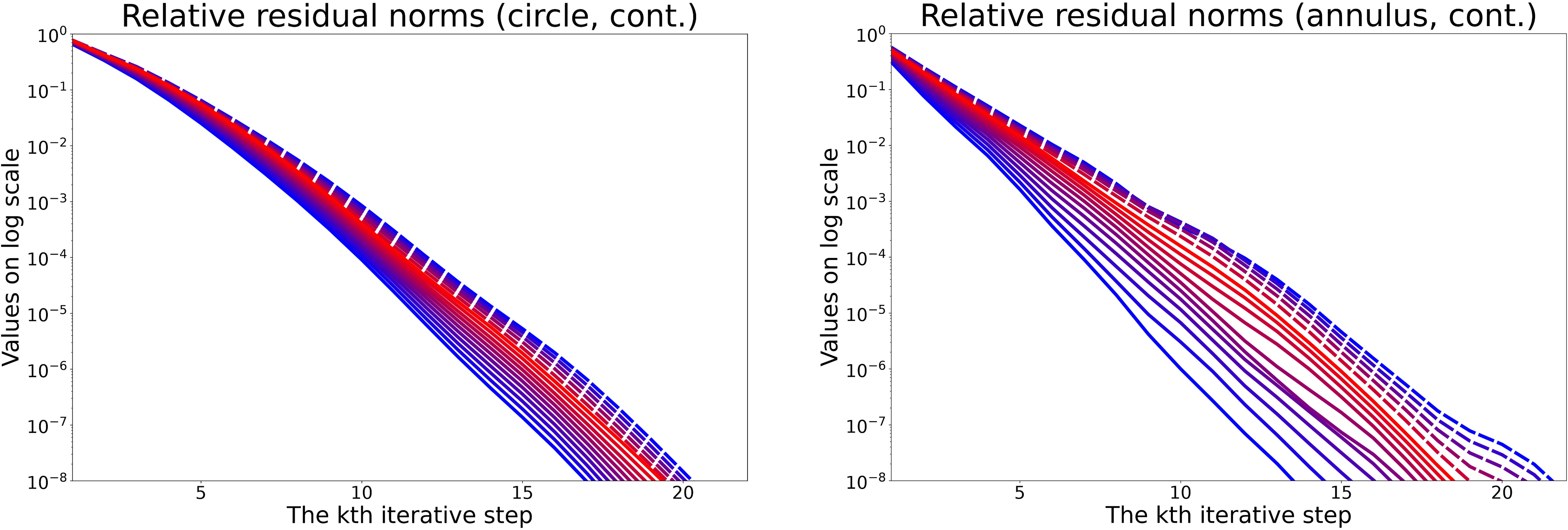}
	\caption{Relative residual norms for the inner (solid) and outer (dashed) continuous approximation of the unit disk (left) and the annulus (right).}
	\label{circle-annulus-cont-graphs}
\end{figure}

\subsubsection{Discrete approximation of an annulus from both sides}
Let $\Omega$ be the annulus determined by the unit circle and the circle of radius 3 centered at the origin. For the inner approximation, let $\Omega_n \;(n \geq 3)$ be the domain between the circumscribed and inscribed regular polygon of the inner and outer circle such that they have $n$ vertices and a vertex on the positive side of the horizontal axis of the coordinate system. For the outer approximation, consider the enlarged versions of these centered at the origin such that they approximate the annulus from the outside similarly to the previous constructions. For a visual representation see Figure \ref{annulus-disc}. The results of the numerical tests with $\beta = 3$ can be seen in Figure \ref{circle-annulus-disc-graphs}.

\begin{figure}[t!]
	\centering
	\includegraphics[width=1\textwidth]{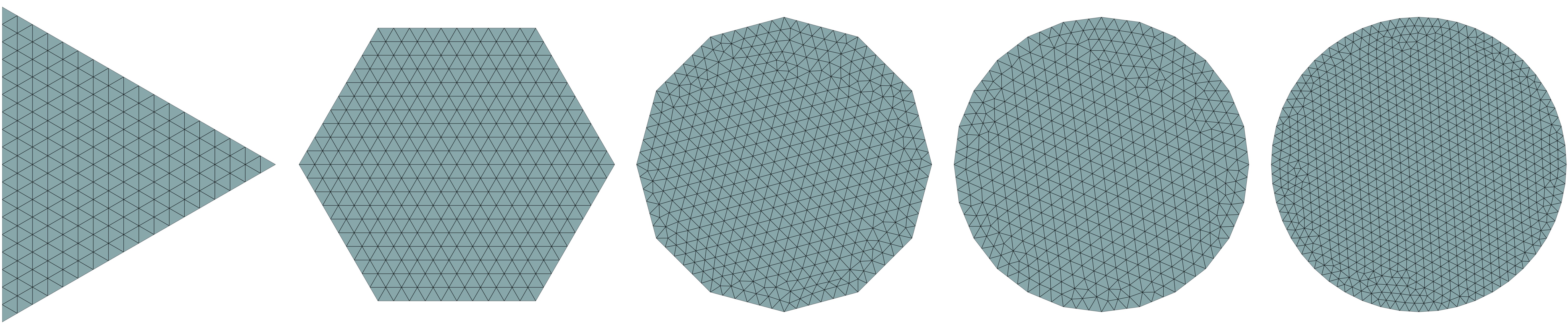}
	\caption{Discrete monotonic approximation of the unit disk with inscribed regular polygons of $3,\;6,\;12,\;24,\;48$ vertices from left to right.}
	\label{circle-disc}
\end{figure}
\begin{figure}[t!]
	\centering
	\includegraphics[width=1\textwidth]{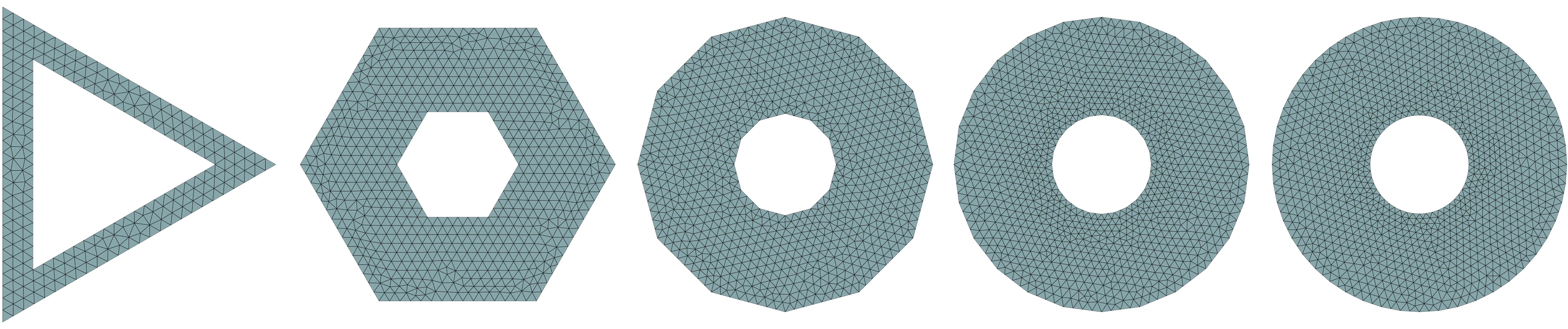}
	\caption{Discrete monotonic approximation of the annulus with regular polygons of $3,\;6,\;12,\;24,\;48$ vertices from left to right.}
	\label{annulus-disc}
\end{figure}
\begin{figure}[h!]
	\centering
	\includegraphics[width=1\textwidth]{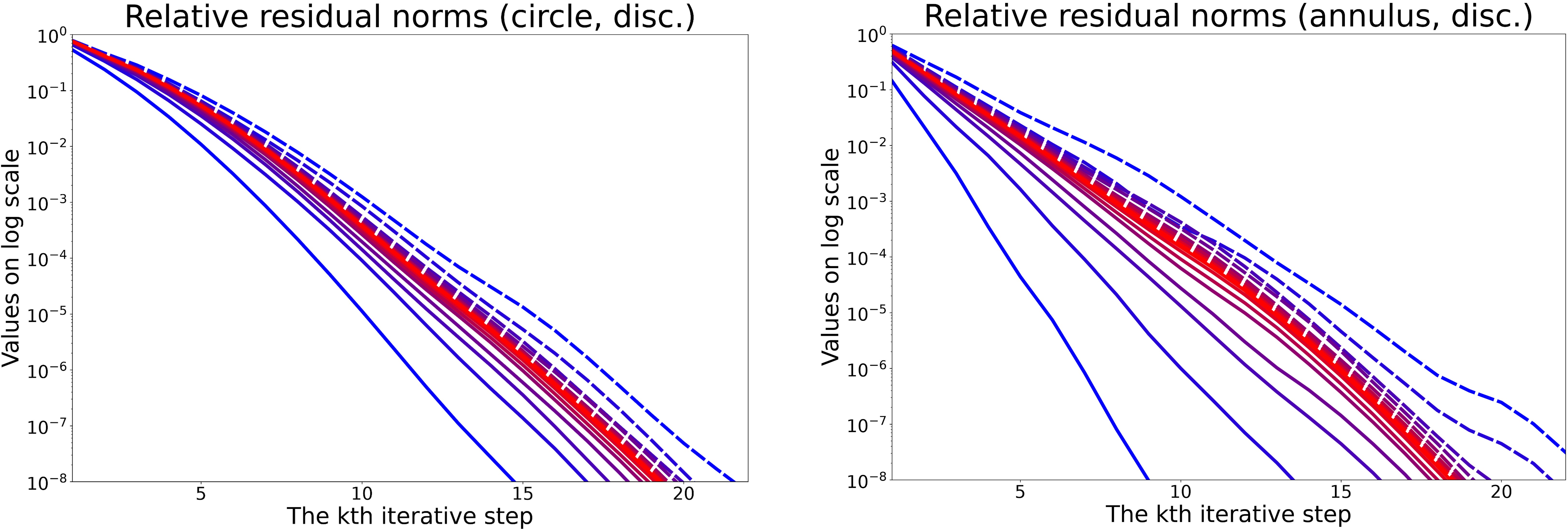}
	\caption{Relative residual norms for the inner (solid) and outer (dashed) discrete approximation of the unit disk (left) and the annulus (right).}
	\label{circle-annulus-disc-graphs}
\end{figure}

\subsubsection{Discrete monotonic approximation of an L-shaped domain}
In this example, let $\Omega$ be an L-shaped domain having a concave corner at the origin, which is interesting as it can cause numerical difficulties in general. This domain is approximated by first removing the quadrant of a circle from the upper right corner of the square, and then replacing the concave corner in the middle with a circular arc tangent to the sides of the L-shaped domain at its endpoints. For a visual representation see Figure \ref{L-shape-disc}. The results of the numerical tests with $\beta = 8$ can be seen in Figure \ref{L-shape-kidney-disc-graphs}.

\subsubsection{Discrete approximation of a kidney-shaped domain}

Let $\Omega_n \;(n \geq 3)$ be a polygonal domain with $n$ vertices uniformly placed on a predefined smooth kidney-shaped closed curve which defines the approximated domain $\Omega$. In this example, instead of having a distinct inner or outer approximation, the direction of approximation depends on whether the boundary of $\Omega$ is locally convex or concave at a given region. For a visual representation see Figure \ref{kidney-disc}. The results of the numerical tests with $\beta = 7$ can be seen in Figure \ref{L-shape-kidney-disc-graphs} when $n=4, 7, 9, 11, 40$.

\begin{figure}[h!]
	\centering
	\includegraphics[width=1\textwidth]{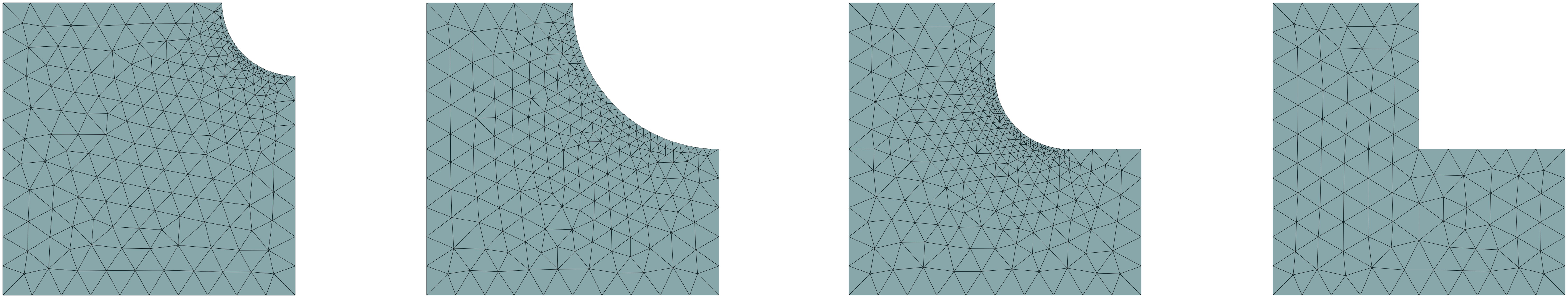}
	\caption{Discrete monotonic approximation of the L-shaped domain.}
	\label{L-shape-disc}
\end{figure}
\begin{figure}[h!]
	\centering
	\includegraphics[width=1\textwidth]{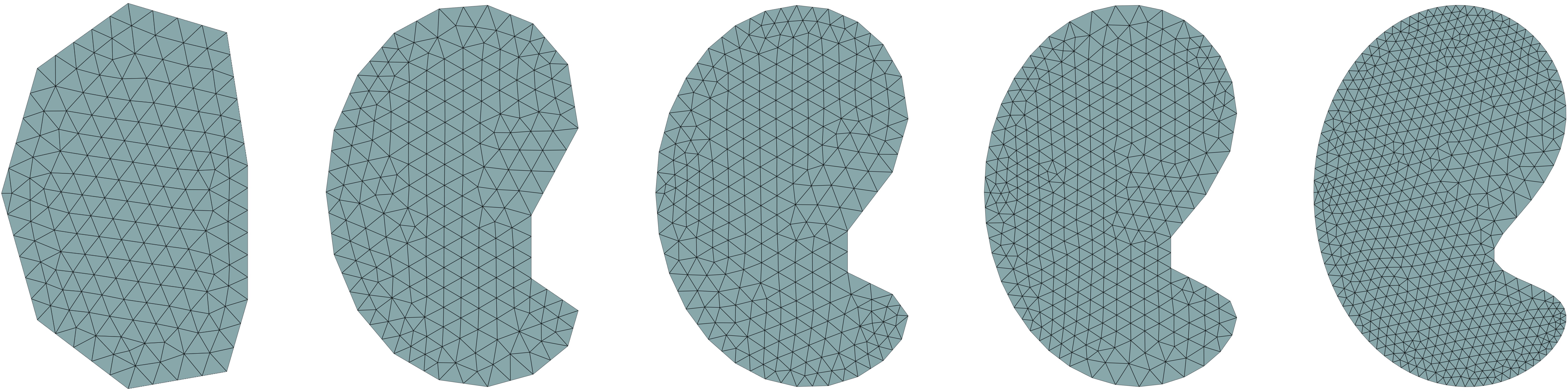}
	\caption{Discrete approximation of the kidney-shaped domain with $10,\;20,\;30,\;40,\;100$ vertices from left to right.}
	\label{kidney-disc}
\end{figure}
\begin{figure}[h!]
	\centering
	\includegraphics[width=1\textwidth]{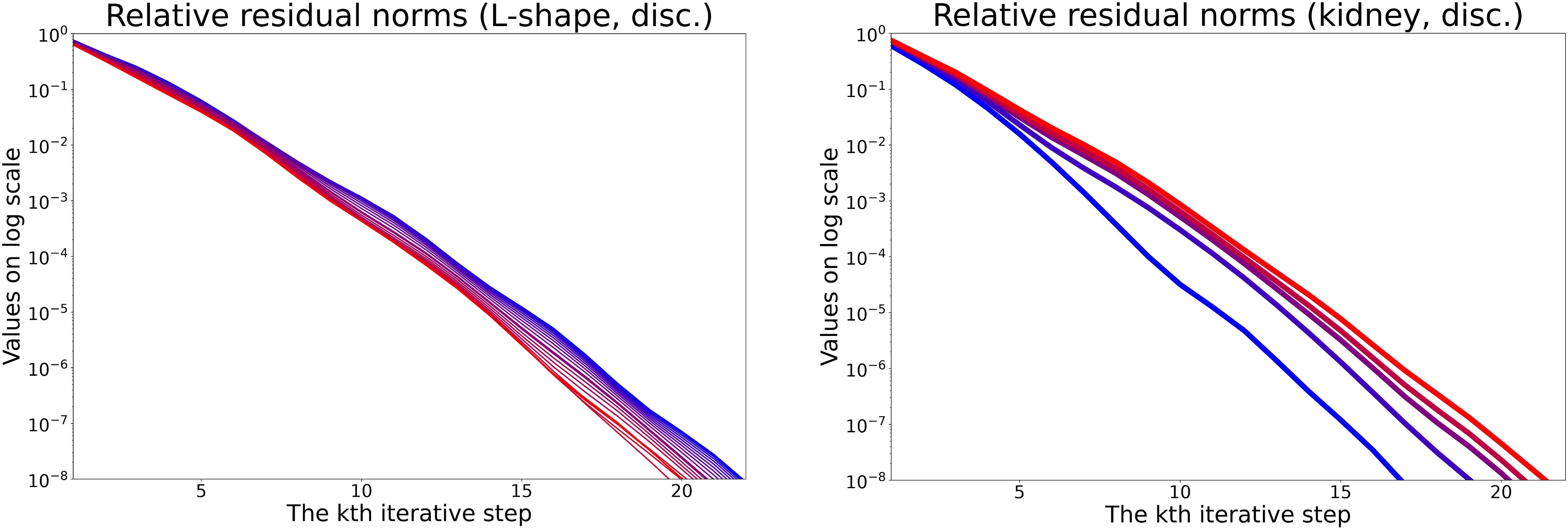}
	\caption{Relative residual norms for the discrete approximation of the L-shaped and the kidney-shaped domain.}
	\label{L-shape-kidney-disc-graphs}
\end{figure}

\subsection{The Koch snowflake and the Ball Pit}
The iterative construction of the Koch snowflake was described in section \ref{ss:examples_conv_not}. Let $\Omega_n$ be its iterations. For a visual representation of the Koch snowflake, see Figure \ref{fractal-disc}. The results of the numerical tests with $\beta = 13$ can be seen in Figure \ref{ball-pit-disc-graphs} (left). We can observe the continuous monotonic dependence of the curves on the domain, which was predicted by the theoretical considerations.

\subsubsection{Non-continuous dependence in the Ball Pit}
In this section, we present the numerical experiments on the Ball Pit domain which was analyzed in section \ref{ss:examples_conv_not}. In this case, we know that there exist eigenvalues of the Laplacian which do not depend continuously on the domain, so we expect that the relative residual norms also do not converge to the norms corresponding to the limit set.

For a visual representation of the Ball Pit, see Figure \ref{ball-pit}. The results of the numerical tests with $\beta = 8$ can be seen in Figure \ref{ball-pit-disc-graphs} (right). We can observe that apart from the first few iterative steps, the curves do not converge to the relative residual norms corresponding to the limit set, which is visualized as a red dashed curve in the figure. This behavior is consistent with the theoretical findings, and it reveals that continuous dependence is not automatic also for the discretized problems in general.

\begin{figure}[h!]
	\centering
    \includegraphics[width=1\textwidth]{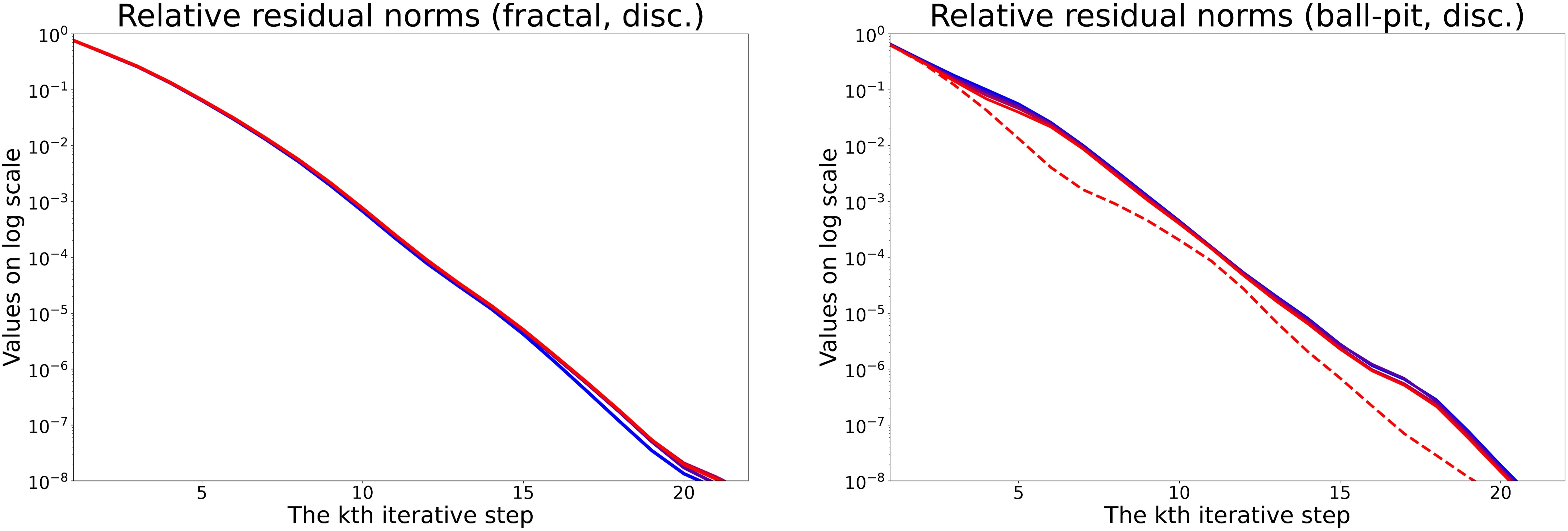}
	\caption{Relative residual norms for the first few iterations of the Koch snowflake (left) and for the iterations of the Ball Pit (right) when $n=3,4,5,6$ (blue-red solid curves) and for the limit set (red dashed curve).}
	\label{ball-pit-disc-graphs}
\end{figure}

\subsection{Conclusions and future directions}
\label{subsconcf}

In this paper we have shown that the analysis of the iterative solution of a class of PDEs can be characterized with the eigenvalues of the Laplacian.
We have also identified geometric conditions (convexity, uniform Lipschitz regularity, or a topological restriction in the plane) under which this quantity depends continuously on the domain, on both the PDE and the discrete level.
In particular, the continuity established for domains bounded by a Koch snowflake is reflected in the discrete case, and so is the failure of continuity for the perforated domains of Example~\ref{ball-pit-ex}.
Altogether, we have seen that this kind of meeting of the continuous and discrete levels is really fruitful: continuity questions  on the discrete level, where the underlying matrices change with the mesh, can be understood
once transferred to the operator level, and the conclusions transfer back in the behaviour of the discrete problems.

Our study used convection-diffusion type problems as an important class of stationary PDEs, since such equations model several fundamental physical phenomena. We note, on the other hand, that the main point in reducing the convergence bounds to eigenvalues substantially exploited the decomposition of the underlying operator in the form
$$
I+K,
$$
where $K$ is compact.

Hence we expect that similar results to ours may be elaborated for some other remarkable models where the  operator has a similar structure, often in the concrete form $\frac{1}{2}I+K$:
\begin{itemize}
\item Systems of  elliptic PDEs,  without or with convection terms, using decoupled preconditioners, where the components of $K$ are lower order operators as in the present paper, see  \cite{KA05, eqopsurv}.
\item The Dirichlet problem can be recast to a boundary integral equation using a double layer potential, leading to the form $\frac{1}{2}I+K$ with a compact boundary integral operator $K$, see, e.g., \cite{CH89}.
\item A similar situation as above holds when the 
Maxwell transmission problem is described with the Calderón operator \cite{Ces96}.
\item The Schur complement operator for the Stokes problem on smooth domains can be reduced to a similar structure, exploited in the superlinear convergence estimation of the  Krylov--Uzawa iteration \cite{AK21}.
\end{itemize}
An analogous study of continuous geometry  dependence in such problems might be the subject of future work.

\noindent {\bf Acknowledgements.}
The research of J. Kar\'atson was supported by the Hungarian National Research, Development and Innovation Fund (NKFIH) under the scheme ELTE TKP 2021-NKTA-62 and the grant no. K137699. 

\pagestyle{empty}

\providecommand{\bysame}{\leavevmode\hbox to3em{\hrulefill}\thinspace}
\providecommand{\MR}{\relax\ifhmode\unskip\space\fi MR }
\providecommand{\MRhref}[2]{%
  \href{http://www.ams.org/mathscinet-getitem?mr=#1}{#2}
}
\providecommand{\href}[2]{#2}

\end{document}